\documentclass[12pt]{amsart}
\usepackage{times}
\usepackage[normalem]{ulem}
\usepackage[T1]{fontenc}
\usepackage{dsfont}
\usepackage{mathrsfs}
\usepackage[colorlinks]{hyperref}
\usepackage{xcolor}
\usepackage[colorlinks]{hyperref}
\usepackage{xcolor}
\usepackage[a4paper,asymmetric]{geometry}
\usepackage{mathscinet}
\usepackage{fullpage}
\usepackage[a4paper,asymmetric]{geometry}
\usepackage{mathscinet}
\usepackage{latexsym}
\usepackage{amsthm}
\usepackage{amssymb}
\usepackage{amsfonts}
\usepackage{amsmath}
\usepackage{longtable}
\usepackage{graphicx}
\usepackage{multirow}
\usepackage{multicol}
\usepackage[utf8]{inputenc}
 \usepackage[T1]{fontenc}
 \usepackage{mathtools}
  \usepackage{bm}

\newtheorem{theorem}{Theorem}[section]
\newtheorem{thm}[theorem]{Theorem}

\newtheorem{lem}[theorem]{Lemma}
\newtheorem{proposition}[theorem]{Proposition}
\newtheorem{prop}[theorem]{Proposition}
\newtheorem{corollary}[theorem]{Corollary}

\theoremstyle{definition}

\newtheorem{defn}[theorem]{Definition}

\theoremstyle{remark}
\newtheorem{remark}[theorem]{Remark}
\newtheorem{rem}[theorem]{Remark}
\numberwithin{equation}{section}

 \DeclareMathAlphabet{\mathpzc}{OT1}{pzc}{m}{it}
\newcommand\redsout{\bgroup\markoverwith{\textcolor{red}{\rule[0.5ex]{2pt}{1.5pt}}}\ULon}

\newcommand{\dif}{\mathrm{d}}

\newcommand{\abs}[1]{\left\vert#1\right\vert}
\newcommand{\set}[1]{\left\{#1\right\}}

\newcommand{\norm}[1]{\left\Vert#1\right\Vert}

 \DeclareMathOperator{\tr}{tr}

\newcommand{\E}{\mathbb{E}}
\newcommand{\PP}{\mathbb{P}}

\newcommand{\e}{\varepsilon}

\newcommand{\Ll}{\langle}
\newcommand{\Rr}{\rangle}

\newcommand{\N}{\mathbb{N}}
\newcommand{\R}{\mathbb{R}}

\newcommand{\C}{\mathbb{C}}
\newcommand{\mcl}{\mathcal}

\newcommand{\Be}{\begin{equation}}
\newcommand{\Ee}{\end{equation}}
\newcommand{\Bs}{\begin{split}}
\newcommand{\Es}{\end{split}}
\newcommand{\Bes}{\begin{equation*}}
\newcommand{\Ees}{\end{equation*}}
\newcommand{\BT}{\begin{thm}}
\newcommand{\ET}{\end{thm}}
\newcommand{\Bp}{\begin{proof}}
\newcommand{\Ep}{\end{proof}}
\newcommand{\BL}{\begin{lem}}
\newcommand{\EL}{\end{lem}}
\newcommand{\BP}{\begin{proposition}}
\newcommand{\EP}{\end{proposition}}
\newcommand{\BC}{\begin{corollary}}
\newcommand{\EC}{\end{corollary}}
\newcommand{\BR}{\begin{rem}}
\newcommand{\ER}{\end{rem}}
\newcommand{\BD}{\begin{defn}}
\newcommand{\ED}{\end{defn}}
\newcommand{\BI}{\begin{itemize}}
\newcommand{\EI}{\end{itemize}}

\newcommand{\tl}{\tilde}

\newcommand{\mi}{{\rm i}}
 \newcommand{\innp}[1]{\langle {#1}\rangle}
\newcommand{\cle}{\mathcal{E}}
\definecolor{amet}{rgb}{0.8, 0.2, 0.8}

\definecolor{ques}{rgb}{0.8, 0.2, 0.2}

\begin{document}
\title[LDP and FT for the EPR of Gaussian Process]{Large Deviations of the Entropy Production Rate for a Class of Gaussian Processes
}
\author[A. Budhiraja]{Amarjit Budhiraja}
\address{Department of Statistics and Operations Research,
University of North Carolina,
Chapel Hill, NC 27599, United States}
\email{budhiraj@email.unc.edu}
\author[Y. Chen]{Yong Chen}
\address{College of Mathematics and Information Science, Jiangxi Normal University,
Nanchang, Jiangxi,  330022, P. R. China }
\email{zhishi@pku.org.cn}
\author[L. Xu]{Lihu Xu}
\address{1. Department of Mathematics, Faculty of Science and Technology
University of Macau Av. Padre Tom\'{a}s Pereira, Taipa Macau, China \\
2. UMacau Zhuhai Research Institute, China}
\email{lihuxu@um.edu.mo}
\begin{abstract}
We prove a large deviation principle (LDP) and a fluctuation theorem (FT) for the entropy production rate (EPR) of the following $d$ dimensional stochastic differential equation
\begin{equation*}  
\dif X_{t}=AX_{t} \dif t+\sqrt{Q}\dif B_{t} 
\end{equation*}
where $A$ is a real normal stable matrix, $Q$ is  positive definite, and the matrices $A$ and $Q$ commute.  The rate function for the EPR takes the following explicit form:
\begin{equation*}
I(x)=\left\{
\begin{array}{ll}
x\frac{\sqrt{1+\ell_0(x)}-1}{2}+\frac 12\sum\limits_{k=1}^{d}  \left(\sqrt{\alpha_k^2- \beta_k^2\ell_0(x)}+\alpha_k\right) ,  & x\ge 0 ,  \\
-x\frac{\sqrt{1+\ell_0(x)}+1}{2} +\frac 12\sum\limits_{k=1}^{d} \left(\sqrt{\alpha_k^2- \beta_k^2\ell_0(x)}+\alpha_k\right)  , &x<0,
\end{array}
\right.
\end{equation*} where $\alpha_{k}\pm {\rm i} \beta_{k}$ are the eigenvalues of $A$, and 
$\ell_0(x)$ is the unique solution of the equation:
\begin{align*}
{\abs{x}}={\sqrt{1+\ell}} \times \sum_{k=1}^{d} \frac{\beta_k^2}{\sqrt{\alpha_k^2 -\ell\beta_k^2} },\qquad -1 \le \ell<  \min_{k=1,...,d}\set{\frac{\alpha_k^2}{\beta_k^2}}.
\end{align*}
Simple closed form formulas for rate functions are rare and our work identifies an important class of large deviation problems where such formulas are available.
The logarithmic moment generating function (the fluctuation function) $\Lambda$ associated with the LDP is given as  
\Bes
\Lambda(\lambda)=\begin{cases} 
	-\frac 12\sum\limits_{k=1}^{d}   \left(\sqrt{ \alpha_k^2- 4\lambda(1+\lambda) \beta_k^2} +\alpha_k\right) \ \ \ &  \lambda \in \mcl D, \\
	\infty, & \lambda \notin \mcl D,
\end{cases}
\Ees 
where $\mcl D$ is the domain of $\Lambda$. The functions
 $\Lambda(\lambda)$ and $ I(x)$ satisfy the  Cohen-Gallavotti symmetry properties:
	\begin{equation*} 
	\Lambda(x)=\Lambda(-(1+x)),\quad I(x)=I(-x)-x,\, \mbox{ for all } x\in \R.
	\end{equation*}
	
In particular, the functions $I$ and $\Lambda$ do not depend on the diffusion matrix $Q$, and are determined completely by the real and imaginary parts of the eigenvalues of $A$. Formally, the deterministic system with $Q=0$ has zero EPR and thus the model exhibits a phase transition in that the EPR changes discontinuously at $Q = 0$.\\


\end{abstract}

\maketitle

\section{Introduction}  \label{s:Intr}
Let $\{X_t\}$ be a Markov process with sample paths in $C([0, \infty); \cle)$ (space of continuous functions from
$[0, \infty)$ to 
$\cle$ equipped with the local uniform topology), where $\cle$ is some Polish space, with a stationary distribution
$\mu$. For $t>0$, denote by $\PP^\mu_{[0,t]}$ the probability law on $C([0,t]; \cle)$ (space of continuous functions from
$[0, t]$ to 
$\cle$ equipped with the uniform topology) of $\{X_s\}_{0\le s \le t}$, under the stationary measure (i.e. when $X_0$ has distribution $\mu$). Also, denote by $\PP^{\mu,-}_{[0,t]}$ the probability law of the time reversed process, namely $\PP^{\mu,-}_{[0,t]} = \PP^\mu_{[0,t]} \circ \theta_t^{-1}$, where $\theta_t: C([0,t]; \cle) \to
C([0,t]; \cle)$ is defined as $\theta_t(x)(s)\doteq x(t-s)$, $0\le s \le t$.
Sample entropy production rate, also known as the Gallavotti-Cohen functional, of the Markov process $\{X_t\}$
under the stationary distribution $\mu$ is defined as
\[ e_p(t) = \begin{cases}
                    \frac{1}{t}\log \frac{\dif \PP^\mu_{[0,t]}}{\dif \PP^{\mu,-}_{[0,t]}} & \mbox{ if }
					    \PP^\mu_{[0,t]} \ll \PP^{\mu,-}_{[0,t]}\\
                     \infty & \mbox{ otherwise } 
                 \end{cases}, \; t >0. \]
Note that $e_p(t)=0$ for all $t>0$ if and only if the Markov process is reversible under the stationary measure
$\mu$. Thus EPR can be viewed as a measure of irreversibility of the process $\{X_t\}$.
Entropy production as a means to quantify  irreversibility of a physical
process has been considered in a broad range of model settings e.g. chemical networks \cite{End 17} and biological populations \cite{SeaYad Lin 18}; and across a wide spectrum of temporal and spatial scales,  from cells to planetary climates \cite{Whit 05}. The mathematical formulation of EPR for stochastic processes in terms of time reversed processes originated in {\cite{Kur, LeSp, QM91}}; see also
%
%
%
\cite{JQQ2004, LanToOli 13, ToOli 12, ECM, GC95}. 
For more recent work on asymptotics of entropy production rate, we refer the reader to \cite{BG15, San Lan Peter 17, San Ce Lan 19, JPS16, JPS17} and the references therein.
In this work we consider a $\R^d$-valued Markov process given by the solution of the stochastic differential equation (SDE):
\begin{equation}  
\dif X_{t}=AX_{t} \dif t+\sqrt{Q}\dif B_{t} \label{langevin0}
\end{equation}
where $A \in \R^{d \times d}, Q \in \R^{d \times d}$ satisfy
 the following stability and irreducibility assumption: 
\begin{itemize}
\item[{\bf (A)}]  All the eigenvalues of $A$ have negative real parts and $Q$ is positive definite.
\end{itemize}
Under Assumption  {\bf (A)}, $\{X_t\}$ admits a unique invariant measure $\mu$ and  the empirical EPR process $e_p(t)$ is well defined. The first goal of this work is to establish a large deviation principle (LDP) for the
EPR process $e_p(t)$ as $t\to \infty$.  From the classical work of Donsker and Varadhan\cite{DoVa7-8} (see also \cite{Che08,DoVa7-8,Kif90,Wu01}), the large deviation behavior of the empirical measure process $\frac{1}{t}\int_0^t \delta_{X_s} ds$ is well understood. However, the large deviations  of the process $e_p(t)$ cannot be deduced from these results in a simple manner. In particular, this process is given in terms of empirical average of a quadratic functional of the state process together with a time averaged stochastic integral (see \eqref{eprform}). The superlinearity of the functional of interest makes the analysis of large deviation properties of the EPR particularly challenging. 

A second objective of this work is to establish the `fluctuation theorem' (FT) for the diffusion given in \eqref{langevin0}. A fluctuation theorem in non-equilibrium statistical mechanics, as formulated in the mathematical theory by Cohen and Gallavotti\cite{GC95}, is a symmetry property of the rate function $I$ associated with the LDP for the EPR process $e_p(t)$ which states that $I(x) + \frac{x}{2} = I(-x) - \frac{x}{2}$, for all $x \in \R^d$. Formally speaking, such a property gives a universality result which says that, for large $T$, the ratio between the probabilities of events $\{e_T= x\}$ and $\{e_T = -x\}$ is close to a model independent quantity given as
$$\frac{P[e_T=x]}{P[e_T=-x]} \approx \exp\{-T(I(x)- I(-x))\} = \exp\{Tx\}.$$
In this work,  the fluctuation theorem that gives the Cohen-Gallavotti symmetry properties of the rate function, for the model in \eqref{langevin0}, will be established by identifying an explicit closed form expression for the rate function (see \eqref{rate func}). Explicit formulas for  rate functions are rare and our work identifies an interesting and important class of large deviation problems for which such formulas are available.

A special case of \eqref{langevin0} was studied in
Chen et al. \cite{cgxx} where a three dimensional Langevin equation governing the motion of a charged test particle in a constant magnetic field
was analyzed.
In this model $X_{t}=[X^{1}_{t}, X^{2}_{t},X^{3}_{t}]'$,  with $'$ being the transpose operation,  is the particle's velocity at the time $t$; $B_{t}=[B^{1}_{t}, B^{2}_{t},B^{3}_{t}]'$ is a three dimensional standard Brownian motion, and
$$A=-\begin{bmatrix} \cos \theta & - \sin \theta & 0 \\  \sin \theta & \cos \theta & 0 \\ 0 & 0 & \cos \theta
 \end{bmatrix}, \ \ Q=\begin{bmatrix} {\cos \theta} & 0 & 0 \\  0 & {\cos \theta} & 0 \\  0 & 0 & {\cos \theta}\end{bmatrix}, \ \ \ \theta \in (-\frac{\pi}2, \frac{\pi}2).$$

The coordinates  $X^{3}_{t}$ and $[X^{1}_{t}, X^{2}_{t}]'$ are respectively parallel and perpendicular to the magnetic field, and so the Lorentz force only acts on $[X^{1}_{t}, X^{2}_{t}]'$, driving the test particle to spiral around the magnetic field line. The Brownian motion $B_t$ models the collisions between the test particle and the ones in the medium. The dissipation in \eqref{langevin0}, with a strength $\cos \theta$, models the `wave propagation' produced by particles collisions, and the amplitude and sign of $\sin \theta$ describe the strength and direction of the magnetic field respectively. 
Finally, the diffusion coefficient $\cos \theta$ is derived from the classical mean square displacement assumption in statistical mechanics. For additional details on the background of \eqref{langevin0}, we refer the reader to \cite[Chapter 11]{Bal97}.
We note that this three dimensional system satisfies the  following magnetic field property:
\begin{equation}  \label{e:AQComm}
AA'=A'A, \ \ \ \ AQ=QA.
\end{equation}
This property plays a key role in the analysis.
Since in the above three dimensional model $X_{t}^{3}$ is a  reversible Ornstein-Uhlenbeck process (independent of the first two coordinates),  whose EPR is always zero, it suffices, for studying the  large deviation and fluctuation theorem for EPR of \eqref{langevin0}, to consider the simplified two dimensional system given as:
\begin{align}  
\begin{bmatrix}
\dif X_t^{1}\\ \dif X_t^{2} \end{bmatrix}& =-\begin{bmatrix} \cos \theta & - \sin \theta \\  \sin \theta & \cos \theta 
 \end{bmatrix}
\begin{bmatrix} X_t^1\\ X_t^2 \end{bmatrix} \dif t + \begin{bmatrix} {\sqrt{\cos \theta}} & 0  \\  0 & \sqrt{\cos \theta} \end{bmatrix}
\begin{bmatrix} \dif B_t^1\\ \dif B_t^2 \end{bmatrix}. \label{langevin-0}
\end{align}
A natural approach for the study of a large deviation principle for $e_p(t)$ is by an application of the 
G\"artner-Ellis Theorem \cite{demzeit} (see e.g. \cite{BG15,GC95, JPS16, JPS17}). A key step in the implementation of this approach is to compute the 
associated Cram\'er function and study its regularity properties. The paper \cite{cgxx} computed the 
Cram\'er function for the EPR associated with the above reduced two-dimensional problem by studying 
the Karhunen-Lo\`{e}ve (KL) expansion of the complex valued stochastic process $z(t)=X_t^{1}+{\rm i} X_t^{2}$\cite{ito}. In the current work we are interested in the general $d$-dimensional diffusion governed by
\eqref{langevin0} where the coefficients satisfy the stability and irreducibility condition in {\bf (A)} and
the magnetic field property  \eqref{e:AQComm}.
Analogous to the above three dimensional SDE  for a single test particle in a magnetic field, the general $d$-dimensional
equation in \eqref{langevin0} can be viewed as a model for the dynamics of a 
collection of charged test particles in a constant magnetic field (cf. \cite{Bal97}).
In this general setting
a central tool in the proofs of \cite{cgxx}, namely a  Karhunen-Lo\`{e}ve  expansion for $X_t$ with independent stochastic coefficients is not available. {Such an expansion was used in \cite{cgxx} (see equation (14)  therein) in order to represent the
 (pre-limit) Cram\'er function  (see \eqref{e:LamZt}) in terms of an infinite product of the form
 {\begin{equation}
 \label{eq:infprod}
 \prod_{k\ge 1} \E e^{\frac{\lambda(1+\lambda)}{2} w_k^2},	
 \end{equation}
where $\lambda \in \R$} and $\{w_k\}$ is a sequence of independent normal random variables. This representation
played a key role in  \cite{cgxx} in the computation of the 
asymptotics of the Cram\'er function {(see eg. \cite[Corollary 3.3]{cgxx})}.
%
For the general multidimensional setting considered here,
such a simple form representation cannot be given and therefore,  here we take a different approach than the one based on a
Karhunen-Lo\`{e}ve  expansion.
The starting point in this approach is to 
decompose the integral of the quadratic form that appears in the exponent for the (pre-limit) Cram\'er function  (see \eqref{e:LamZt}) into a Wiener-It\^{o} chaos expansion up to the second order. 
The representation that we obtain for the Cram\'er function (see Theorem \ref{t:ExpIntZ}) is significantly more involved than the representation in terms of independent normal random variables in \eqref{eq:infprod} used in \cite{cgxx} and the study of its asymptotics requires a careful analysis of the spectral properties of the integral operator in the second chaos of the 
Wiener-It\^{o}  expansion.
Specifically, a key idea in the analysis is to decompose the {symmetric compact} operator, associated with the map $H_{\lambda,T}$ appearing in the second chaos (see Proposition \ref{p:WienerDec}), acting on the complex Hilbert space $L^2([0,T];\mathbb{C}^d)$, into $d$ operators on $L^2([0,T];\mathbb{C})$, and obtain estimates on its eigenvalues by solving a family of Sturm-Liouville problems. 
{These estimates play a central role in characterizing the domain of finiteness of the Cram\'er functional. This characterization of the domain and an analysis of its boundary properties is the main ingredient in the application of the G\"artner-Ellis Theorem. We remark that Wiener-It\^{o} chaos expansions to study properties of exponential functionals of Gaussian quadratic forms have been used in \cite{GrHu05, GaJi07} in the study of certain mathematical finance problems; however the analysis methods are quite different from those used in the current work.
Also,  analysis of Sturm-Liouville problems to study spectral properties of integral operators has along history, in particular, in a simpler context such an analysis was carried out in \cite{cgxx}.}
Specifically,  \cite{cgxx}  considers integral operators $K_T$ on $L^2([0,T];\mathbb{C})$, $T<\infty$, associated with a single kernel which is the covariance function of the complex Gaussian process $z(t)$, whereas here we need to analyze the spectral properties of  integral operators  on
$L^2([0,T];\mathbb{C}^d)$ associated with an infinite collection of kernels $\{H_{\lambda, T}, \lambda \in \R, T <\infty\}$ (see Section \ref{s:EigK}).}

A notable feature of our results is that, {although the invariant measure of the Markov process in \eqref{langevin0}
obviously depends on $Q$,} the Cram\'{e}r function and the rate function for the LDP of $e_p(t)$
do not depend on this matrix (as long as Assumption {\bf (A)} and \eqref{e:AQComm} are satisfied). These quantities are completely determined from the real and imaginary parts of the eigenvalues of $A$, see Theorems \ref{t:EPR} and \ref{t:FT}. The noise term in \eqref{langevin0} is usually thought of as external heat, with $Q$ describing its temperature, and the rate and Cram\'{e}r functions $I$ and $\Lambda$ can be viewed as a type of entropy \cite{DoVa7-8} and free energy \cite{HaSc07}, respectively.  Our results say that as the temperature tends to zero, the entropy and the free energy quantities identified in Theorems \ref{t:EPR} and \ref{t:FT} do not change. Formally, the deterministic system with $Q=0$, {which is the zero trajectory,} has zero EPR and thus the model exhibits a phase transition in that the EPR changes discontinuously at $Q =0$.
It would be interesting to identify more general conditions under which a stochastic system exhibits such a phenomenon.

%

In a recent work Bertini et al. \cite{BG15} study a $d$-dimensional diffusion process with diffusion coefficient
$\sqrt{\epsilon} \mbox{Id}$ where $\epsilon$ is a small parameter. In view of the technical obstacles arising from the unboundedness (in fact quadratic growth
when the drift is linear) 
of functionals describing the EPR process, they consider a setting where the drift $c$ admits a decomposition
of the form $c = -\frac{1}{2} \nabla V +b$ where $b$ is a smooth and {\em bounded } vector field with bounded
derivatives that is orthogonal to $\nabla V$  at every $x \in \R^d$. They propose a modified definition of the
Gallavotti-Cohen functional (i.e. the sample EPR) that is given in terms of the bounded vector field $b$ (see (1.7) therein). One of the main results in \cite{BG15} establishes a large deviation principle for the modified
Gallavotti-Cohen functional, as  $\epsilon \to 0$ and  $T\to \infty$ (in that order),  which is described in terms of the 
Freidlin-Wentzell quasipotential associated with the small noise asymptotics of the diffusion \cite{FrWe84}.
Although, unlike the current work, \cite{BG15} allows for nonlinear drifts, their analysis, due to the boundedness of $b$, becomes more tractable in some ways as they do not need to handle terms with a quadratic growth in the EPR process. 

  
The recent works of {Jak\v{s}i\'{c}}  et al. \cite{JPS16, JPS17} prove an LDP for the entropy production for a family of Gaussian dynamical systems. Their proofs are analytic, relying on the properties of the maximal solution of a one-parameter family of algebraic matrix Riccati equations, rather than the tools from  stochastic analysis such as chaos decomposition that are used in our work. In particular, their results do not provide explicit closed form formulas for the rate function of the type in \eqref{rate func}.


The paper is organized as follows. Section \ref{s:Epr} introduces some notation and definitions that are frequently used in this work and presents our main results. 
Sections \ref{s:AnaZ} and \ref{s:EigK}  provide the analysis of the integral of the quadratic form that appears in the pre-limit Cram\'{e}r function (see \eqref{e:LamZt})  and that of the eigenvalues of the symmetric compact operator associated with the map $H_{\lambda,T}$ in the  chaos expansion of this integral (see Proposition \ref{p:WienerDec}), respectively. Based on these preparatory results, we prove the LDP and the FT for the  EPR in Section \ref{s:MResult}. 
The Appendix contains the elementary proof of the fact that as a consequence of \eqref{e:AQComm} 
\Be  \label{e:CommQAEtAl}
\mbox{ the collection of matrices }
\mathbb{M} \doteq \{A, A', Q, Q^{1/2}, Q^{-1}, Q^{-1/2}, M, N\} \mbox{ is a commuting family, }
\Ee
namely, for any $F,G \in \mathbb{M}$, $FG=GF$. Here
$$M=A+A^{'}, \ \ \ \ N=A-A^{'}.$$

\section{Entropy production rate and main result} \label{s:Epr}
\subsection{Definition of EPR and the main results}
Throughout this work we assume that
Assumption {\bf (A)} and the magnetic field property \eqref{e:AQComm} are satisfied. 
Recall from Section \ref{s:Intr} that under Assumption {\bf{A}}, $X_t$ admits a unique invariant measure $\mu$.
In fact this measure  is absolutely continuous with respect to
Lebesgue measure and has the following density function:
\begin{equation}  \label{e:Rho}
   \rho({x})=(2\pi)^{-\frac{d}{2}}(\mathrm{det} \Gamma)^{-\frac{1}{2}}\exp\left(-\frac{x^{'}\Gamma^{-1}{x}}2\right),
\end{equation}
with
\begin{equation*}
   \Gamma=\int_0^{\infty}e^{{A}s}{Q}e^{{A}^{'} s}\dif s=-Q M^{-1},
\end{equation*}
where the last equality is  by \eqref{e:CommQAEtAl} from which it follows that
 $AQ=QA$ and $A^{'}Q=Q A^{'}$, and by using the fact that $M$ is negative definite (cf. \cite{wangxu}).

Denote by $(X^{\mu}_{t})_{t \ge 0}$ the stationary process obtained by solving the SDE \eqref{langevin0} with the random
variable $X_0$ distributed according to 
 $\mu$. For $0\le s<t<\infty$, consider the process $(X^{\mu}_{u})_{s \le u \le t}$ and its time reversal $(X^{\mu,-}_{u})_{s\le u\le t}:=(X^{\mu}_{t+s-u})_{s\le u\le t}$. Note that both  are stationary  Markov processes.
Let $\PP^\mu_{[s,t]}$ and $\PP^{\mu,-}_{[s,t]}$ be  the distributions of $(X^{\mu}_{u})_{s\le u\le t}$ and $(X^{\mu,-}_{u})_{s\le u\le t}$.

Occasionally,  for notational simplicity,  we will write
$$X_{t}=X^{\mu}_{t}, \ \ \ \ X^{-}_{t}=X^{\mu,-}_{t}, \ \ \ \ \ t>0.$$

From \cite{JQQ2004}, the sample EPR of Eq. \eqref{langevin0} is
\begin{align}
e_p(t):&=\frac{1}{t}\log \frac{\dif \PP^\mu_{[0,t]}}{\dif \PP^{\mu,-}_{[0,t]}} \notag\\
&=\frac{1}{t}\left[\int_{0}^t\,({Q}^{\frac12}F(X_s))^{'} \dif B_s  +\frac12 \int_0^t |{Q}^{\frac12}F(X_s)|^{2}\dif s\right]  \notag \\
&=\frac{1}{t}\left[\int_{0}^t\,({Q}^{-\frac12} NX_s)^{'} \dif B_s  +\frac12 \int_0^t |{Q}^{-\frac12} NX_s|^2 \dif s\right]
\label{eprform}
\end{align}
where
\begin{equation} \label{fx}
   F({x})=2{Q}^{-1}{A}{x}-\nabla\log \rho({x})=(2{Q}^{-1}{A}+\Gamma^{-1}){x}=Q^{-1} Nx.
\end{equation}
Recall that the process $(X^{\mu}_t)_{t\ge0}$ is  {\bf reversible} if   $\PP^\mu_{[s,t]}=\PP^{\mu,-}_{[s,t]}$  for any $0\le s<t<\infty$. For such a process the sample EPR is simply zero. Thus we will only be concerned with
situations where $X^{\mu}$ is not reversible.
It is well known that the stationary Ornstein-Uhlenbeck process $(X_{t}^{\mu})_{t \ge 0}$ is reversible if and only if the coefficients $A$ and $Q$ satisfy the symmetry condition $Q^{-1/2}A=(Q^{-1/2}A)^{'}$ (cf. \cite[p. 1338]{wangxu}). 
From \eqref{e:CommQAEtAl} it then follows that if $A$ is symmetric then, under our assumptions, the Markov process
$X^{\mu}$ is reversible. Thus we will assume that
$A$ is not symmetric.  Together with the fact that $A$ is a normal matrix, this implies that 
\begin{equation}
	\label{eq:impfact}
	\mbox{ not all of the eigenvalues of } A \mbox{ are real}.
\end{equation}

By the ergodic theorem, we have that,
\
\begin{equation}
 \mbox{ as } t \to \infty, e_p(t) \mbox{ converges a.s.  to } \frac 12 \int_{\R^d}  |{Q}^{-\frac12} N x|^{2} \rho(x) \dif x.
\end{equation}
Our main result gives asymptotics of probabilities of deviations from the above law of large numbers limit by 
establishing the following large deviation principle.
\begin{thm}[Large Deviation Principle]  \label{t:EPR} 
	Suppose that Assumption {\bf (A)} and \eqref{e:AQComm} are satisfied and that $A$ is not a symmetric matrix.
 Then the sample EPR $e_{p}(t)$ satisfies an LDP in $\R$ with the rate function $I$ given as
\begin{equation}\label{rate func}
  I(x)=\left\{
      \begin{array}{ll}
     x\frac{\sqrt{1+\ell_0(x)}-1}{2}+\frac 12\sum\limits_{k=1}^{d} \left(\sqrt{\alpha_k^2-\beta_k^2\ell_0(x)}+\alpha_k\right) ,  & x\ge 0 ,  \\
     -x\frac{\sqrt{1+\ell_0(x)}+1}{2} +\frac 12\sum\limits_{k=1}^{d} \left(\sqrt{\alpha_k^2-  \beta_k^2\ell_0(x)}+\alpha_k\right) , &x<0,
      \end{array}
\right.
\end{equation} where $\alpha_{k}\pm {\rm i} \beta_{k}$ are the eigenvalues of $A$, and 
 $\ell_0(x)$ is the unique solution of the  equation
\begin{align}\label{eq:absxeq}
 {\abs{x}}={\sqrt{1+\ell}} \times \sum_{k=1}^{d} \frac{\beta_k^2}{\sqrt{\alpha_k^2 -\ell \beta_k^2} },\qquad -1 \le \ell<  \min_{k=1,...,d}\set{\frac{\alpha_k^2}{\beta_k^2}},
\end{align}
 namely the following properties hold
  \begin{itemize}
	    \item[(i)] $I$ has compact level sets, i.e. for every $K<\infty$, $\{x \in \R: I(x)\le K\}$ is a compact set.
  \item[(ii)] For each closed subset $F \subset \R$,
              $$
                \limsup_{t\rightarrow \infty}\frac1{t}\log\mathbb P^{\mu} (e_{p}(t)\in F)\leq- \inf_{x\in F}I(x);
              $$
   \item[(iii)] For each open subset $G \subset \R$,
              $$
               \liminf_{t\rightarrow \infty}\frac1{t}\log\mathbb P^{\mu}(e_{p}(t)\in G)\geq- \inf_{x\in G}I(x),
              $$
   \end{itemize}    
  {where $\mathbb P^{\mu}$ is the measure under which $X_0$ is distributed according to $\mu$.}
\end{thm}
{\begin{remark}
	From \eqref{eq:impfact} it follows that $\min_{k=1,...,d}\set{\frac{\alpha_k^2}{\beta_k^2}} <\infty$.
	The uniqueness of solutions of  \eqref{eq:absxeq} is argued in the proof of Theorem \ref{t:EPR} given in Section \ref{ss:PLDP}.
\end{remark}}

In order to prove Theorem \ref{t:EPR} we will apply G\"{a}rtner-Ellis Theorem (cf. \cite{demzeit}). For this we  will need to compute the Cram\'er  function of $e_p(t)$. Define, for $t\ge 0$ and $\lambda \in \R$,
\begin{equation}  \label{e:LamTlam}
\begin{split}
\Lambda_t(\lambda)&:=\frac{1}{t}\log \E^{\mu} \exp\left\{t \lambda e_p(t)\right\} \\
&=\frac{1}{t}\log \E^{\mu}\exp\set{\lambda \int_{0}^t\,({Q}^{-\frac12} NX_s)^{'} \dif B_s  +\frac\lambda 2 \int_0^t |{Q}^{-\frac12} NX_s|^2 \dif s },
\end{split}
\end{equation}
{where $\E^{\mu}$ is the expectation function associated with $\mathbb{P}^{\mu}$.}
For  $\lambda \in \R$,  define
$$\Lambda(\lambda):=\lim_{t \rightarrow \infty} \Lambda_t(\lambda) \ \ \ {\rm provided}
\ \lim_{t \rightarrow \infty} \Lambda_t(\lambda) \  {\rm exists},$$
and 
$$\mathcal D_{\Lambda}:=\{\lambda \in \R: \Lambda(\lambda)<\infty\}.$$
$\Lambda$ is called the Cram\'er function \cite{demzeit} associated with large deviations of $e_p(t)$.

The following theorem gives a closed form expression for $\Lambda$ and its domain and also gives our main fluctuation theorem.
\begin{thm} [Fluctuation Theorem]   \label{t:FT}
Assume the same conditions as in Theorem \ref{t:EPR}.   
Let $\mcl D$ be the closed interval given as
$$\left[ -\frac12-\frac12\sqrt{1+ \min_{k=1,...,d}\set{\frac{\alpha_k^2 }{\beta_k^2}}},\, -\frac12+\frac12 \sqrt{1+\min_{k=1,...,d}\set{\frac{\alpha_k^2 }{\beta_k^2}}}\ \right].$$
Then, we have
\begin{equation}\label{eq:defnlam}
\Lambda(\lambda)=\begin{cases} 
-\frac 12\sum\limits_{k=1}^{d}   \left(\sqrt{ \alpha_k^2- 4\lambda(1+\lambda) \beta_k^2} +\alpha_k\right) \ \ \ &  \lambda \in \mcl D, \\
\infty, & \lambda \notin \mcl D,
\end{cases}
\end{equation}
in particular, $\mathcal D_{\Lambda} = \mcl D$.
Furthermore, $\Lambda$ and $ I$ satisfy the following Cohen-Gallavotti symmetry properties:
\begin{equation}\label{Cohen symmet} 
\Lambda(x)=\Lambda(-(1+x)),\quad I(x)=I(-x)-x,\, \mbox{ for all } x\in \R.
\end{equation}
\end{thm} 
\vskip 3mm


We note that the  results in \cite{cgxx} follow as a special case of Theorems \ref{t:EPR} and \ref{t:FT}. 
Also, as noted in the Introduction, under our assumptions, the Cram\'{e}r function and the rate function for the LDP of $e_p(t)$
do not depend on the matrix $Q$ and since,  formally speaking, the deterministic system with $Q=0$ has zero EPR, the model exhibits a phase transition in that the EPR changes discontinuously at $Q =0$.


\subsection{An auxiliary SDE}   \label{ss:2.2}
From the definition of sample EPR in \eqref{eprform}, and noting from \eqref{e:CommQAEtAl} that $Q^{-\frac12} N=NQ^{-\frac 12}$, we see that
\Bes
\begin{split} 
e_p(t)&=\frac{1}{t}\left[\int_{0}^t\,(N {Q}^{-\frac12}X_s)^{'} \dif B_s  +\frac12 \int_0^t |N {Q}^{-\frac12}X_s|^2 \dif s\right] \\
&=\frac{1}{t}\left[\int_{0}^t\,(N \hat X_s)^{'} \dif B_s  +\frac12 \int_0^t |N \hat X_s|^2 \dif s\right],
\end{split}
\Ees
where $\hat X_{s}={Q}^{-\frac12}X_s$. Observe that $\hat X_{t}$ satisfies the  equation
\Bes
\dif \hat X_{t}=A \hat X_{t} \dif t+\dif B_{t},
\Ees
and the Markov process $\hat X_{t}$ admits a unique invariant measure $\hat \mu$ with a density function 
\begin{equation}\label{eq:density}
\begin{split}
\hat \rho(x)&=(2\pi)^{-\frac{d}{2}}(\mathrm{det}(\hat \Gamma))^{-\frac{1}{2}}\exp\left(-\frac{x^{'}\hat \Gamma^{-1}{x}}2\right)\\
&=(2\pi)^{-\frac{d}{2}}|\mathrm{det} (M)|^{\frac{1}{2}}\exp\left(\frac{x^{'}M{x}}2\right),
\end{split}
\end{equation}
where $\hat \Gamma=\int_0^{\infty}e^{{A}s}e^{{A}^{'} s}\dif s=-M^{-1}$.
Recalling the definition of  $\Lambda_{t}(\lambda)$ from \eqref{e:LamTlam}, we have 
\begin{equation}  
\begin{split}
\Lambda_t(\lambda)
&=\frac{1}{t}\log \E^{\hat \mu}\exp\set{\lambda \int_{0}^t\,(N \hat X_s)^{'} \dif B_s  +\frac\lambda 2 \int_0^t |N \hat X_s|^2 \dif s}. 
\end{split}
\end{equation}
From the above observation we see that, under our assumptions, the quantity $\Lambda_t(\lambda)$ and the distribution of $e_p(t)$ under the stationary distribution of $X$
are independent of the choice of $Q$. This fact explains why the rate function in the LDP and 
the Cram\'{e}r's function in the FT do not depend on $Q$. 
{\em In view of this invariance,
 in rest of this work we assume that $Q={\rm Id}$, i.e. the identity matrix, and, instead of \eqref{langevin0}, consider the equation}
\Be\label{eq:redu}
\dif X_{t}=AX_{t} \dif t+\dif B_{t}.
\Ee
Note that with the new definition, the  invariant measure $\mu$ of $X$ is  the measure $\hat \mu$ defined above.
\vskip 3mm 

In order to compute the Cram\'er function associated with $e_p(t)$, we introduce the following auxiliary equation. For $\lambda \in \R$, let $Y_{\lambda, t}$ solve the equation
\begin{equation}  \label{langevin new}
\begin{split}
\dif Y_{\lambda, t}=D_\lambda Y_{\lambda, t}\dif t +\dif W_t
\end{split}
\end{equation}
where
\Be  \label{e:BLam}
D_\lambda={A}+\lambda N, \ \ \ \ \lambda \in \R.
\Ee
By \eqref{e:CommQAEtAl}, it is clear that $D_{\lambda}D^{'}_{\lambda}=D^{'}_{\lambda}D_{\lambda}$ and 
$$M=D_{\lambda}+D^{'}_{\lambda}, \ \ \ \ \ \ \ \ \Gamma_{\lambda}:=\int_{0}^{\infty}  e^{D_{\lambda} s}e^{D_{\lambda}^{'} s}\dif s=-M^{-1}.$$
Thus for every $\lambda$, the Markov process $(Y_{\lambda,t})_{t \ge 0}$ has the same unique ergodic measure $\mu$  as that of $(X_{t})_{t \ge 0}$ (cf. \cite{wangxu}).
For  $\lambda \in \R$, define
\Be  \label{e:LamZt}
\tilde \Lambda_t(\lambda)=\frac{1}{t}\log \E^{\mu}\exp\set{\frac12\lambda(1+\lambda)\int_{0}^{t}\, |Z_{\lambda, s}|^{2} \dif s},
\Ee
where
\Be  \label{e:ZLams}
Z_{\lambda, s}=N Y_{\lambda, s},
\Ee
and $\E^{\mu}$ denotes the expectation with respect to the probability measure under which $Y_{\lambda, 0}$ is distributed
as $\mu$.
For $\lambda \in \R$, define
$$\tilde \Lambda(\lambda):=\lim_{t \rightarrow \infty} \tilde \Lambda_t(\lambda) \ \ \ {\rm provided} \ \lim_{t \rightarrow \infty} \tilde \Lambda_t(\lambda) \ {\rm exists},$$
and
$$\mathcal D_{\tilde \Lambda}:=\{\lambda \in \R: \tilde \Lambda(\lambda)<\infty\}.$$
We shall show in Proposition \ref{p:PropTilD1} that 
\Be  \label{e:TilD=D}
\mathcal D_{\tilde \Lambda}=\mathcal D_{\Lambda}, \ \ \ \ \ \ \ {\tilde \Lambda}={\Lambda}.
\Ee


\section{The analysis of $\int_{0}^T\,|Z_{\lambda, s}|^{2} \dif s$} \label{s:AnaZ}
In this section we will decompose  $\int_{0}^T\,|Z_{\lambda, s}|^{2} \dif s$ into its Wiener-It\^{o} chaos expansion 
  and compute its exponential moments by analyzing an operator associated with  its second order chaos.  
  For a function $\psi \in L^2([0,T]^2; \R^{d\times d})$, we write
  $\int_{[0,T]^2} dW'_{u_1} \psi(u_1, u_2) dW_{u_2}$ for the $\R$ valued random variable 
  defined as
  $$\sum_{i,j=1}^d\int_{0}^T \left(\int_0^t \psi_{ij}(s, t) dW_{s}^i\right) dW_{t}^j
  + \sum_{i,j=1}^d\int_{0}^T \left(\int_0^s \psi_{ij}(s, t) dW_{t}^j\right) dW_{s}^i.$$
Stochastic integrals of the form
$$\int_0^{u_1}   dW'_{u_1} \psi(u_1, u_2), \; \int_0^{u_2}    \psi(u_1, u_2) dW_{u_2}$$
are $\R^d$ valued random variables interpreted in a similar manner.
  
\begin{prop} \label{p:WienerDec}
	Let, for $\lambda \in \R$,
	$Y_{\lambda,t}$ and $Z_{\lambda,t}$ be as defined in \eqref{langevin new} and \eqref{e:ZLams} respectively.
Suppose that for some $x\in \R^d$,  $Y_{\lambda,0}=x$. Then, for all $T\ge 0$, we have
   \begin{equation}\label{hts}
      \int_{0}^T\,|Z_{\lambda, s}|^{2} \dif s=S_{0}^x(T)+S_{\lambda,1}^x(T)+S_{\lambda,2}(T),
   \end{equation}
   where
   \begin{equation}\label{eq:soxt}
	   S_{0}^x(T)=\left|\left(\int_0^T e^{sM} \dif s\right)^{\frac 12}       N x\right|^{2}+\int_{0}^{T} \tr\left[N^{'} e^{uM} {N}\right] (T-u) \dif u,\end{equation}
   \Bes
   \begin{split}
   & S_{\lambda,1}^x(T)=\int_0^T (G_{\lambda,T}^{x}(u))' \dif W_{u},  \ \ \ \  \ \ \ \ \ S_{\lambda,2}(T)=\frac12 \int_{[0,T]^2} \dif W_{u_1}^{'} H_{\lambda,T}(u_1,u_2) \dif W_{u_2},
   \end{split}
   \Ees
   with  $$G^x_{\lambda,T}(u)=2  N^{'} e^{-u {D^{'}_\lambda}} \left(\int_{u}^{T} e^{Ms} \dif s\right)       {N}x,$$
   \begin{align*}
      H_{\lambda,T}(u_1,u_2)&=2 \left(e^{- u_1 {D_\lambda}} N\right)^{'} \left(\int_{u_1\vee u_2}^T\,e^{ {M}t}\dif t\right) e^{- u_2 {D_\lambda}} N.
   \end{align*}
\end{prop}
\begin{proof}  
   From (\ref{langevin new}), for all $s\ge 0$, 
   \begin{equation*}
      Y_{\lambda, s}=e^{s {D_\lambda}}x+ \int_0^s e^{(s-u) {D_\lambda}}  \dif W_u,
    \end{equation*}
 which implies that
   \begin{equation*}
    Z_{\lambda, s}=        {N} Y_{\lambda, s}=       {N}e^{s {D_\lambda}}x+ \int_0^s e^{(s-u) {D_\lambda}} {N}\dif W_u.
   \end{equation*}
   where we have used $N D_{\lambda}=D_{\lambda} N$ (which is a consequence of \eqref{e:CommQAEtAl}). 
Thus,
   \begin{equation}
   |Z_{\lambda, s}|^{2}
    =|      {N}e^{s {D_\lambda}}x|^{2}+2   (      {N}e^{s {D_\lambda}}x)^{'}\int_0^s e^{(s-u) {D_\lambda}} N\dif W_u+\left|\int_0^s e^{(s-u) {D_\lambda}} {N}\dif W_u\right|^{2}.  \label{e:Z^2}
   \end{equation}
  Let us first consider the last term on the right hand side of \eqref{e:Z^2}.
   By \eqref{e:CommQAEtAl} the matrices $N$, $D_\lambda$ and $ D^{'}_{\lambda}$ commute and thus   recalling that $D^{'}_\lambda+D_\lambda=M$, we have   
    \begin{align*}
  \left|\int_0^s e^{(s-u) {D_\lambda}} {N}\dif W_u\right|^{2}&= \left(\int_0^s e^{-u D_\lambda} {N}\dif W_u\right)^{'} e^{s(D^{'}_\lambda+D_\lambda)} \left(\int_0^s e^{-u {D_\lambda}} {N}\dif W_u\right)  \\
 &=\left(\int_0^s e^{-u D_\lambda} {N}\dif W_u\right)^{'} e^{s M} \left(\int_0^s e^{-u {D_\lambda}} {N}\dif W_u\right).
   \end{align*}
For fixed $s$, applying It\^{o}'s formula to the semimartingale
 $$r\mapsto \left(\int_0^r e^{-u D_\lambda} {N}\dif W_u\right)^{'} e^{s M} \left(\int_0^r e^{-u {D_\lambda}} {N}\dif W_u\right),$$  we have
   \begin{align*}
& \ \ \left(\int_0^r e^{-u D_\lambda} {N}\dif W_u\right)^{'} e^{s M} \left(\int_0^r e^{-u {D_\lambda}} {N}\dif W_u\right) \\
&=2\int_{0}^{r} \left(\int_0^u e^{-v D_\lambda} {N}\dif W_v\right)^{'} e^{s M}  e^{-u {D_\lambda}} {N}\dif W_u+\int_{0}^{r} \tr[N^{'} e^{-u D^{'}_\lambda} e^{s M} e^{-u {D_\lambda}} {N}] \dif u \\
&=2\int_{0}^{r} \left(\int_0^u e^{-v D_\lambda} {N}\dif W_v\right)^{'} e^{s M}  e^{-u {D_\lambda}} {N}\dif W_u+\int_{0}^{r} \tr[N^{'} e^{(s -u)M} {N}] \dif u,
   \end{align*}
   where the last equality is once more from the fact that the matrices $N$,  $D_\lambda$ and $ D^{'}_{\lambda}$ commute, and that $D^{'}_{\lambda} +D_{\lambda}=M$. Taking $r=s$ in the above expression yields
   \begin{align}\label{eq:tresterms}
  \left|\int_0^s e^{(s-u) {D_\lambda}} {N}\dif W_u\right|^{2}&=2\int_{0}^{s} \left(\int_0^u e^{-v D_\lambda} {N}\dif W_v\right)^{'} e^{s M}  e^{-u {D_\lambda}} {N}\dif W_u+\int_{0}^{s} \tr[N^{'} e^{(s -u)M} {N}] \dif u.
   \end{align}
 The integral of the first  term on the right hand side of \eqref{eq:tresterms} can be written as
 \begin{align*}
 &  \ \ \ 2\int_{0}^{T}\int_{0}^{s} \left(\int_0^u e^{-v D_\lambda} {N}\dif W_v\right)^{'} e^{s M}  e^{-u {D_\lambda}} {N}\dif W_u \dif s \\
 &=\int_{0}^{T}\int_{0}^{s} \left(\int_0^u e^{-v D_\lambda} {N}\dif W_v\right)^{'} e^{s M}  e^{-u {D_\lambda}} {N}\dif W_u \dif s\\
 &\quad +\int_{0}^{T}\int_{0}^{s} \left(\int_0^v e^{-u D_\lambda} {N}\dif W_u\right)^{'} e^{s M}  e^{-v {D_\lambda}} {N}\dif W_v \dif s \\
  &=\int_{0}^{T} \left(\int_0^u \dif W_v^{'}e^{-v D_\lambda^{'}} {N^{'}}\right) \left(\int_{u}^{T} e^{s M} \dif s\right) e^{-u {D_\lambda}} {N}\dif W_u \\
  &\ \ \ \ \ +\int_{0}^{T} 
  \dif W_v^{'}\left[\left(\int_{v}^{T}e^{s M} \dif s\right) e^{-v {D_\lambda^{'}}} {N^{'}}
  \left(\int_0^v e^{-u D_\lambda} {N}\dif W_u\right)\right]  \\
  &=\frac12 \int_{[0,T]^2} \dif W_{u_1}^{'} H_{\lambda,T}(u_1,u_2) \dif W_{u_2},
   \end{align*}
where the  second equality  is by the (stochastic) Fubini Theorem (cf. \cite[Section 3.7]{karshr}).
The integral of the second term on the right side of \eqref{eq:tresterms} can be written as
   \begin{align*}
  \int_{0}^{T}\int_{0}^{s} \tr[N^{'} e^{(s -u)M} {N}] \dif u \dif s= \int_{0}^{T}\int_{0}^{s} \tr[N^{'} e^{u M} {N}] \dif u \dif s= \int_{0}^{T} \tr\left[N^{'} e^{uM} {N}\right] (T-u) \dif u.
   \end{align*}
 Combining the above observations the integral of the left side of \eqref{eq:tresterms} is given as
  \begin{align*}
  \int_{0}^{T}\left|\int_0^s e^{(s-u) {D_\lambda}} {N}\dif W_u\right|^{2} \dif s=\frac12 \int_{[0,T]^2} \dif W_{u_1}^{'} H_{\lambda,T}(u_1,u_2) \dif W_{u_2}+\int_{0}^{T} \tr\left[N^{'} e^{uM} {N}\right] (T-u) \dif u.
   \end{align*}
For the first two terms on the right hand of \eqref{e:Z^2},
integrating over $s\in[0,T]$,
we have
   \begin{align*}
  \int_0^T |      {N}e^{s {D_\lambda}}x|^{2} \dif s&= \int_0^T x^{'} e^{sD^{'}_{\lambda}} N^{'}            N e^{sD_{\lambda}} x \dif s \\
  &=\int_0^T x^{'} N^{'}      e^{sM}       N x \dif s=\left|\left(\int_0^T e^{sM} \dif s\right)^{\frac 12}       N x\right|^{2}
   \end{align*}
   and
    \begin{align*}
   2\int_{0}^{T}(      {N}e^{s {D_\lambda}}x)^{'}\int_0^s e^{(s-u) {D_\lambda}}N \dif W_u \dif s&=2\int_{0}^{T}(      {N}x)^{'} e^{sM}\int_0^s e^{-u {D_\lambda}}N \dif W_u \dif s \\
   &=2\int_{0}^{T} (      {N}x)^{'} \left(\int_{u}^{T} e^{Ms} \dif s\right) e^{-u {D_\lambda}}N \dif W_{u} \\
   &=2\int_{0}^{T} \left[N^{'} e^{-u {D^{'}_\lambda}} \left(\int_{u}^{T} e^{Ms} \dif s\right)       {N}x\right]^{'}  \dif W_{u},
\end{align*}
where the second equality once more uses the (stochastic) Fubini theorem.
Combining the previous three relations with \eqref{e:Z^2}, we immediately get the desired identity in the proposition.
\end{proof}

Let $T>0$ and $\lambda \in \R$. Recall the function $H_{\lambda,T}: [0,T] \times [0,T] \rightarrow \R^{d \times d}$  defined in Proposition \ref{p:WienerDec}. Define the operator $K_{\lambda,T}: L^2([0,T]; \R^d) \rightarrow L^2([0,T]; \R^d)$ by
\Be \label{e:KOperator}
K_{\lambda,T} f(t)=\int_0^T H_{\lambda,T}(t,s) f(s) \dif s, \ \ \ \ \  \ f \in L^2([0,T];\R^d),
\Ee
\begin{lem}  \label{l:KTEig}
$K_{\lambda,T}$ is a nonnegative, symmetric, trace class  operator on $L^2([0,T]; \R^d)$. Consequently there exist a standard orthonormal basis $\{e_n\}_{n \ge 1}$ of $L^{2}([0,T];\R^{d})$ and a real  sequence (possibly depending on
$\lambda$ and $T$) $\gamma_{1} \ge \gamma_{2}  \ge ... \ge \gamma_n \ge ...\ge 0$,  such that
\Be  \label{e:KTEig}
\begin{split}
& H_{\lambda,T}(s,t)={\sum_{n\ge 1}} \gamma_{n} e_n(s) \otimes e_n(t), \quad 0 \le s, t\le T,\\
& K_{\lambda,T} e_n=\gamma_n e_n, \ \ \ \ \ \ n=1,2,...,
\end{split}
\Ee
where for $a,b \in \R^d$, $a \otimes b$ denotes a $d \times d$ matrix  whose $(i,j)$-th entry is $a_ib_j$.
Moreover, we have
\Be  \label{e:TrKTMer}
\begin{split}
{\rm tr}(K_{\lambda,T})&=\int_{0}^{T} {\rm tr}\left(H_{\lambda,T}(u,u)\right) \dif u \\
&=2\left[{\rm tr}(N^{'} M^{-1} (e^{MT}-{\rm Id}) M^{-1} N)-T{\rm tr}(N^{'} M^{-1} N)\right].
\end{split}
\Ee
\end{lem}
\begin{proof}
The symmetry of $K_{\lambda,T}$ is immediate from the symmetry of the function $H_{\lambda,T}$ (namely the property
$H_{\lambda,T}(s,t)= H_{\lambda,T}(t,s)$ for $(s,t) \in [0,T]^2$).
 {Now we show that $K_{\lambda, T}$ is nonnegative, i.e. for any  $f \in L^2([0,T];\R^d)$, the following relation holds:
\Be  \label{e:NonKT}
\int_{[0,T]^{2}}  (f(u_{1}))' H_{\lambda,T} (u_{1},u_{2}) f(u_{2}) \dif u_{1} \dif u_{2} \ge 0.
\Ee}
Since $M$ is a symmetric matrix, there exists some orthogonal matrix $P$ such that
$$P M P^{'}= \mbox{diag}\{\eta_1,...,\eta_d\} = \mbox{diag} (\bm{\eta})  \ \  {\rm with} \ \  \eta_i \in \R \ {\rm for}  \ \ 1 \le i \le d.$$
Therefore, for any $f \in L^2([0,T];\R^d)$, denote $g(t)=Pe^{-t D_{\lambda}} N f(t)$,
we have
\Be\label{eq:eq3.8}
\begin{split}
\int_{[0,T]^{2}}  (f(u_{1}))' H_{\lambda,T} (u_{1},u_{2}) f(u_{2}) \dif u_{1} \dif u_{2} &=2\int_{[0,T]^{2}}  (g(u_{1}))'  \left(\int_{u_{1} \vee u_{2}}^T e^{\mbox{diag}(\bm{\eta}) t}  \dif t\right) g(u_{2}) \dif u_{1} \dif u_{2}  \\
&=2\sum_{i=1}^{d}  \int_{[0,T]^{2}}  g_{i}(u_{1})  \left(\int_{u_{1} \vee u_{2}}^T e^{\eta_{i} t}  \dif t\right) g_{i}(u_{2}) \dif u_{1} \dif u_{2},
\end{split}
\Ee
where $g_{i}(t)$ is the $i$-th element of $g(t)$. 
For  $1 \le i \le d$, 
\Bes
\begin{split}
 & \ \ \ \ \ \ \int_{[0,T]^{2}}  g_{i}(u_{1})  \left(\int_{u_{1} \vee u_{2}}^T e^{\eta_{i} t}  \dif t\right) g_{i}(u_{2}) \dif u_{1} \dif u_{2} \\
 &=2 \int_{0}^{T} \int_{0}^{u_{2}}  g_{i}(u_{1})
 \dif u_{1} \left(\int_{u_{2}}^T e^{\eta_{i} t}  \dif t\right) g_{i}(u_{2}) \dif u_{2} \\
 &=\frac 2{\eta_{i}} \left[\int_{0}^{T} \int_{0}^{u_{2}}  g_{i}(u_{1})
 \dif u_{1} e^{\eta_{i} T}g_{i}(u_{2}) \dif u_{2}-\int_{0}^{T} \int_{0}^{u_{2}}  g_{i}(u_{1})
 \dif u_{1} e^{\eta_{i} u_{2}}g_{i}(u_{2}) \dif u_{2}\right]\\
\end{split}
\Ees
Letting $F_{i}(t)=\int_{0}^{t} g_{i}(s) \dif s$,  by a straightforward calculation, we get
\Bes
\begin{split}
& \int_{0}^{T} \int_{0}^{u_{2}}  g_{i}(u_{1})
 \dif u_{1} e^{\eta_{i} T}g_{i}(u_{2}) \dif u_{2}=\frac12 e^{\eta_{i} T} F^{2}_{i}(T), \\
 & \int_{0}^{T} \int_{0}^{u_{2}}  g_{i}(u_{1}) \dif u_{1} e^{\eta_{i} u_{2}}g_{i}(u_{2}) \dif u_{2}=
 \frac12 e^{\eta_{i} T} F^{2}_{i}(T)-\frac{\eta_{i}}2 \int_{0}^{T} e^{\eta_{i} t} F^{2}_{i}(t) \dif t,
\end{split}
\Ees
thus,
\Bes
\begin{split}
 & \int_{[0,T]^{2}}  g_{i}(u_{1})  \left(\int_{u_{1} \vee u_{2}}^T e^{\eta_{i} t}  \dif t\right) g_{i}(u_{2}) \dif u_{1} \dif u_{2}= \int_{0}^{T} e^{\eta_{i} t} F^{2}_{i}(t) \dif t.
\end{split}
\Ees
Using the above identities in \eqref{eq:eq3.8}, we have the desired inequality in  \eqref{e:NonKT}.

Since $K_{\lambda, T}$ is a nonnegative symmetric operator,  by Mercer's Theorem \cite[Theorem 16.7.1]{Gar07}, there exist a complete orthonormal system $\{e_n\}_{n \ge 1}$ in
$L^2([0,T];\R^{d})$ and a sequence of nonnegative reals $\{\gamma_n\}_{n \ge 1}$ such that $\gamma_{1} \ge \gamma_{2} \ge ....$ and that
\eqref{e:KTEig} holds.
Moreover,
\Bes  
\begin{split}
\int_{0}^{T} {\rm tr}\left(H_{\lambda,T}(u,u)\right) \dif u 
&=2 \int_{0}^{T} {\rm tr}\left[N^{'} e^{- u {D^{'}_\lambda}} \left(\int_{u}^T\,e^{ {M}t}\dif t\right) e^{- u {D_\lambda}} N\right] \dif u \\
&=2 \int_{0}^{T} {\rm tr}\left[N^{'} e^{-u M} \left(\int_{u}^T\,e^{ {M}t}\dif t\right) N\right] \dif u \\
&=2 \int_{0}^{T} {\rm tr}\left[N^{'}  M^{-1}\left(e^{ {M}(T-u)}-{\rm Id}\right) N\right] \dif u \\
&=2\left[{\rm tr}(N^{'} M^{-1} (e^{MT}-{\rm Id}) M^{-1} N)-T{\rm tr}(N^{'} M^{-1} N)\right],
\end{split}
\Ees
where the second equality is by using the the commuting property of matrices  $M$, $D_\lambda$ and $D'_\lambda$, and the fact that $D_\lambda+D'_\lambda=M$. This shows that $K_{\lambda,T}$ is a trace class operator and that ${\rm tr}(K_{\lambda,T})$ is given by
\eqref{e:TrKTMer}.
\end{proof}

Lemma \ref{l:KTEig} shows that $K_{\lambda,T}$ is a symmetric compact (in fact trace class) operator with the spectrum $\sigma(K_{\lambda,T})$ given as 
\Be  \label{e:SigKT}
\sigma(K_{\lambda,T})=\{\gamma_{1}, \gamma_{2},...\}
\Ee 
 and  ${\rm tr}(K_{\lambda,T}) = \sum_{i=1}^{\infty}\gamma_i <\infty$.
\begin{thm} \label{t:ExpIntZ}
	Let, for $\lambda \in \R$,
$Y_{\lambda,t}$ and $Z_{\lambda,t}$ be as defined in \eqref{langevin new} and \eqref{e:ZLams} respectively.
Suppose that for some $x\in \R^d$,  $Y_{\lambda,0}=x$.
Then for any $\theta\in \left(-\infty, \frac{1}{\gamma_{1}}\right)$, where  $\gamma_{i}$ are as in Lemma \ref{l:KTEig}, we have
\begin{equation}\label{suanyi}
\begin{split}
   &\E^{x} \exp\left(\theta\int_{0}^{T} |Z_{\lambda,s}|^{2} \dif s\right)\\
  &= \frac{1}{\sqrt{\mathrm{det}({\rm Id}-\theta K_{\lambda,T})}}  \\
    & \quad \times \exp\left[\theta S_{0}^x(T)-\frac \theta 2  \tr(K_{\lambda,T})+\frac{\theta^2}2 \Ll G_{\lambda,T}^x, ({\rm Id}-\theta K_{\lambda,T})^{-1}G_{\lambda,T}^x\Rr_{L^{2}([0,T]; \R^{d})}\right],
\end{split}
\end{equation}
where $S_{0}^x$ and $G_{\lambda,T}^x$ are as in Proposition \ref{p:WienerDec} and 
\Be  \label{e:DetRep}
\  \frac{1}{\sqrt{\mathrm{det}({\rm Id}-\theta K_{\lambda,T})}}=\prod_{n=1}^\infty \frac{1}{\sqrt{1-\theta \gamma_n}}. 
\Ee
 Moreover,  for all $\theta \ge \frac{1}{\gamma_{1}}$, we have
\begin{equation}\label{suanyi-1}
\begin{split}
   \E^{x}\exp\left(\theta\int_{0}^{T}  |Z_{\lambda,s}|^{2} \dif s\right)=\infty.
\end{split}
\end{equation}
\end{thm}


\begin{proof}
For notational simplicity,  we drop the subscript $x$ in $\E^{x}$. Denote 
$$W(e_{n})=\int_{0}^{T} (e_{n}(s))' \dif W_{s},$$
where $\{e_n\}$ is as in Lemma \ref{l:KTEig}.
It is easy to verify that 
$$\E[W(e_{n})W(e_{m})]=\int_{0}^{T} (e_{n}(s))' e_{n}(s) \dif s=\delta_{mn}, \mbox{ for } m,n \in \mathbb{N}$$
and hence $W(e_{1}), W(e_{2}),...$ is a sequence of i.i.d. standard normal  random variables.

From \eqref{e:KTEig}, we have
\begin{equation*}
\begin{split}
 \int_{[0,T]^2}\dif W_{u_1}^{'} H_{\lambda,T}(u_1,u_2)\dif W_{u_2}&=
  \sum_{n=1}^\infty \gamma_{n} \int_{[0,T]^2}\dif W_{u_1}^{'} e_n(u_1) \otimes e_n(u_2) \dif W_{u_2}, \\
\end{split}
\end{equation*}
where the series converges in $L^2(\PP^x)$. Also by definition of the multiple Wiener-It\^{o} integrals,
\Bes
\begin{split}
\int_{[0,T]^2}\dif W_{u_1}^{'} e_n(u_1) \otimes e_n(u_2) \dif W_{u_2}
&=2 \int_0^T (e_n(u_1))'\left(\int_0^{u_1} (e_n(u_2))' \dif W_{u_2} \right) \dif W_{u_1}.
\end{split}
\Ees
Applying It\^{o} formula to $(\int_0^t (e_n(s))' \dif W_s)^2$, we get
\begin{equation*}
\begin{split}
\left(\int_0^T (e_n(s))' \dif W_s\right)^2
&=2 \int_0^T (e_n(u_1))'\left(\int_0^{u_1} (e_n(u_2))' \dif W_{u_2} \right) \dif W_{u_1}+\int_0^T |e_n(s)|^{2} \dif s.
\end{split}
\end{equation*}
This, together with the fact $\int_0^T |e_n(s)|^{2} \dif s=1$, immediately implies
\begin{equation}  \label{e:DouInt}
\begin{split}
 \int_{[0,T]^2}\dif W_{u_1}^{'} H_{\lambda,T}(u_1,u_2)\dif W_{u_2}&=
  \sum_{n=1}^\infty \gamma_{n} \left[(W(e_{n}))^{2}-1\right].
\end{split}
\end{equation}
In rest of the proof we suppress $x$ in the notation $S_0^x$ and $G_{\lambda,T}^x$.
Since $G_{\lambda,T} \in L^2([0,T];\R^d)$, we can represent it as  $G_{\lambda,T}(t)=\sum_{n=1}^\infty g_{\lambda,n} e_n(t)$
with  the series converging in $L^2([0,T];\R^d)$ and $(g_{\lambda,n})_{n \ge 1}$ being a sequence of real numbers with $\sum_{n} g^{2}_{\lambda,n}<\infty$. Henceforth, we  suppress the subscript $\lambda$ in $g_{\lambda,n}$. Then
$$
\int_0^T (G_{\lambda,T}(t))' \dif W_t=\sum_{n=1}^\infty g_{n} W(e_n).
$$
Hence, combining the above two identities, we have from Proposition \ref{p:WienerDec}
$$\int_{0}^{T} |Z_{\lambda,s}|^{2} \dif s\ =\ S_{0}^x(T)+\sum_{n=1}^\infty g_{n} W(e_n)+\frac 12\sum_{n=1}^\infty \gamma_n \left[(W(e_{n}))^{2}-1\right].$$

Let $N \in \N$. Since $W(e_{1}), W(e_{2}),...$ are i.i.d.  standard normal random variables, if $\theta \gamma_{1}<1$ we have 
\begin{equation*}
\begin{split}
& \ \ \ \ \ \E \exp\left(\theta \sum_{n=1}^N g_{n} W(e_n)+\frac\theta 2 \sum_{n=1}^N \gamma_n \left\{[W(e_n)]^2-1\right\}\right)  \\
&=\exp\left(-\frac \theta 2 \sum_{n=1}^N \gamma_n\right) \prod_{n=1}^N \E\exp\left(\theta g_n W(e_n)+\frac{\theta\gamma_n} 2 [W(e_n)]^2\right) \\
&=\left[\prod_{n=1}^N \frac{1}{\sqrt{1-\theta \gamma_n}}\right] \exp\left(-\frac \theta 2 \sum_{n=1}^N \gamma_n+\frac{\theta^2} 2 \sum_{n=1}^N \frac{g^2_{n}}{1-\theta \gamma_n}\right).
\end{split}
\end{equation*}
Recall that $\sum_{n=1}^{\infty} \gamma_n<\infty$ and $\sum_{n=1}^{\infty} g^{2}_{n}<\infty$, whence 
$$\prod_{n=1}^\infty \frac{1}{\sqrt{1-\theta \gamma_n}}<\infty, \quad \quad \sum_{n=1}^N \frac{g^2_{n}}{1-\theta \gamma_n} \le \sum_{n=1}^N \frac{g^2_{n}}{1-\theta \gamma_1} \le \sum_{n=1}^\infty \frac{g^2_{n}}{1-\theta \gamma_1}<\infty,$$
this implies
 $$\lim_{N \rightarrow \infty}\left[\prod_{n=1}^N \frac{1}{\sqrt{1-\theta \gamma_n}}\right] \exp\left(-\frac \theta 2 \sum_{n=1}^N \gamma_n+\frac{\theta^2} 2 \sum_{n=1}^N \frac{g^2_{n}}{1-\theta \gamma_n}\right)\ \ {\rm exists}.$$
Thus we have shown that, for all $\theta < 1/\gamma_1$, 
\begin{equation}  \label{e:WC-1}
\begin{split}
&\ \ \ \ \lim_{N \rightarrow \infty} \E \exp\left(\theta \sum_{n=1}^N g_n W(e_n)+\frac\theta 2 \sum_{n=1}^N \gamma_n \left\{[W(e_n)]^2-1\right\}\right)  \\
&=\frac{1}{\sqrt{ {\rm det}({\rm Id}-\theta K_{\lambda,T})}} \exp\left(-\frac \theta 2  \tr(K_{\lambda,T})+\frac{\theta^2} 2   \Ll G_{\lambda,T}, ({\rm Id}-\theta K_{\lambda,T})^{-1}G_{\lambda,T}\Rr_{L^{2}([0,T]; \R^{d})}\right)
\end{split}
\end{equation}
where we have used the observation that
\begin{align}\label{key point 1}
\Ll G_{\lambda,T}, ({\rm Id}-\theta K_{\lambda,T})^{-1}G_{\lambda,T}\Rr_{L^{2}([0,T]; \R^{d})}=\sum_{n=1}^\infty \frac{g^2_{n}}{1-\theta \gamma_n}.
\end{align}
In order to complete the proof of the first statement in the theorem it suffices to show that in \eqref{e:WC-1} the order of
 the limit $\lim_{N \rightarrow \infty}$ and $\E$  can be interchanged. For this it is enough to show that for all
 $\theta < 1/\gamma_1$
\begin{equation}  \label{e:WC-2}
 \sup_{N \ge 1} \E \exp\left(\theta \sum_{n=1}^N g_n W(e_n)+\frac\theta 2 \sum_{n=1}^N \gamma_n \left\{[W(e_n)]^2-1\right\}\right)  < \infty
\end{equation}
This follows on observing that the above expectation is bounded above by
$$
 \frac{1}{\sqrt{ {\rm det}({\rm Id}-\theta K_{\lambda,T})}} \exp\left(\frac{|\theta|} 2  \tr(K_{\lambda,T})+\frac{\theta^2} 2 
\sum_{n=1}^\infty \frac{g^2_{n}}{1-\theta \gamma_n}\right),$$
which is clearly finite for every $\theta<\frac{1}{\gamma_{1}}$. This proves the first statement in the theorem.
The second statement is an immediate consequence of the fact that for $\theta \ge 1/\gamma_1$,
$\E \exp\{\theta \gamma_1 (W(e_1))^2\} = \infty$.
\end{proof}
\section{The eigenvalues of the operator $K_{\lambda,T}$}  \label{s:EigK}
  
 Recall from Lemma \ref{l:KTEig} that the operator $K_{\lambda,T}: L^{2}([0,T]; \R^{d}) \rightarrow L^{2}([0,T]; \R^{d})$ is symmetric and has  nonnegative eigenvalues $\gamma_{1} \ge \gamma_{2} \ge ... \ge 0$. This section is devoted to analyzing these eigenvalues. To this end, we extend the operator $K_{\lambda,T}$, in a natural fashion, from $L^{2}([0,T]; \R^{d})$ to $L^{2}([0,T]; \mathbb C^{d})$, namely the complex valued $L^{2}$ function space with inner product: 
 $$\Ll f,g\Rr_{L^{2}([0,T]; \mathbb C^{d})}=\int_{0}^{T} \sum_{i=1}^{d} f_{i}(t) \overline{g_{i}(t)} \dif t, \ \ \ \ \ \ \ f, g \in L^{2}([0,T]; \mathbb C^{d}).$$
 Note that
  $K_{\lambda,T}$ is also a symmetric operator on $L^2 ([0,T]; \mathbb{C}^{d})$ and that $\{e_n\}_{n \ge 1}$ is also an orthonormal basis in $L^{2}([0,T]; \mathbb C^{d})$. This extension will allow us to use tools from complex function theory. For notational simplicity, from now on we denote 
 $$\mcl H=L^2 ([0,T]; \mathbb{C}^{d}).$$
For a complex matrix $V$, $V^*$ will denote its conjugate transpose. Note that when the matrix $V$ is real $V^* = V'$. 
  
 Since $A$ is a normal matrix, there
 is a complex unitary matrix $U$ (cf. \cite[Thmeorem 2.5.8]{hj}) such that
 \begin{equation}\label{ubub}
 {U}^* {A} {U} =\mbox{diag}\set{a_1,\dots,a_d},
 \end{equation}
 where $a_i \in \mathbb{C}$.

Write $U=[U_{1},...,U_{d}]$ with $U_{k}$ being the column vectors of $U$. These vectors form an orthonormal basis of $\mathbb C^{d}$. Denote by $P_{U_k}$ the orthogonal projection from $\mathbb C^{d}$ to the subspace spanned by the vector $U_{k}$, which can be represented as $P_{U_k}=U_{k} \otimes U^{*}_{k}$ under the basis $U_{1},...,U_{d}$. From this it is easily seen that $P_{U_{i}}=P_{U_{i}}^{*}$ for each $i$, $P_{U_{i}} P_{U_{j}}=0$ for $i \ne j$, and $\mathrm{Id}=\sum^{d}_{k=1} P_{U_k}$. It follows from the above spectrum calculation that,
with $a_k = \alpha_k { +}\mi\beta_k$,
\begin{equation}\label{eq:spect}
\begin{aligned}
    & {{A}}=\sum^{d}_{k=1}  a_k P_{U_k}, \ \ \ \ \  \ \ \ \ \  A' = {{A}}^{*}=\sum^{d}_{k=1}  \bar{a}_k P_{U_k}, \\
    & {D}_{\lambda}= \sum^{d}_{k=1}  (a_k+\mi 2\lambda { {\beta}}_k  ) P_{U_k},\ \ \ \ \ \ D_{\lambda}^{*}= \sum^{d}_{k=1}  (\bar{a}_k-\mi 2\lambda   { {\beta}}_k ) P_{U_k},\\
    &{M}=  {{A}}+  {{A}}^{*}=\sum^{d}_{k=1}  2\alpha_k P_{U_k},\ \ \ \  \ \ {N}=   {{A}}-  {{A}}^{*}=\sum^{d}_{k=1}  2\mi\,\beta_k P_{U_k},
\end{aligned}
\end{equation}
and that
\begin{align}\label{e: S01}
\left|\left(\int_0^T e^{sM} \dif s\right)^{\frac 12}       N x\right|^{2}=\sum^{d}_{k=1} \frac{2\beta_k^2}{\alpha_k} (e^{\alpha_k T}-1)\abs{\innp{x, U_k}_{\C^d}}^2
\end{align}
and 
\begin{align}
 G^x_{\lambda,T} (u)&=2  N^{'} e^{-u {D^{'}_\lambda}} \left(\int_{u}^{T} e^{Ms} \dif s\right)       {N}x \nonumber \\
&=\sum^{d}_{k=1} \frac{4\beta_k^2}{\alpha_k}   e^{-(\bar{a}_k-\mi 2\lambda  { {\beta}}_k)u} \big(e^{2\alpha_k T}-e^{2\alpha_k u} \big) \innp{x,\,U_k}_{\mathbb C^{d}} U_k,\label{e:GxT}\\
H_{\lambda,T}(u_{1},u_{2})&=2 \left(e^{- u_1 {D_\lambda}} N\right)^{*} \left(\int_{u_1\vee u_2}^T\,e^{ {M}t}\dif t\right) e^{- u_2 {D_\lambda}} N  \nonumber \\
   &=\sum^{d}_{k=1} \frac{-4\beta_k^2}{\alpha_k} e^{-(\bar{a}_k-\mi 2\lambda  { {\beta}}_k)u_1-(a_k+\mi 2\lambda  { {\beta}}_k)u_2}(e^{2\alpha_k(u_1\vee u_2) }-e^{2\alpha_k T})\, P_{U_k},  \label{e:HDec}
\end{align}
{where for $x,y \in \mathbb C^{d}$, $\innp{x,\,y}_{\mathbb C^{d}} := \sum_{i=1}^d x_i \bar y_i$.}
Let $\mcl H_{k}=P_{U_{k}} \mcl H$ where $P_{U_{k}} \mcl H=\{P_{U_{k}} f: f \in \mcl H\}$. Clearly $\mcl H=\bigoplus_{i=1}^{d} \mcl H_{k}$, i.e., 
\Be  \label{e:DirSumF}
f=\oplus_{k=1}^{d} P_{U_{k}} f, \ \ \ \ \ \ f \in \mcl H.
\Ee
By \eqref{e:KOperator} and \eqref{e:HDec}, for any $f \in \mcl H$, we have 
\begin{equation}  \label{e:KTDecF}
\begin{split}
   K_{\lambda,T}f(u_{1})
   &= \sum^{d}_{k=1} \frac{-4\beta_k^2}{\alpha_k}\int_0^Te^{-(\bar{a}_k-\mi 2\lambda  { {\beta}}_k)u_1-(a_k+\mi 2\lambda  { {\beta}}_k)u_2}(e^{2\alpha_k(u_1\vee u_2) }-e^{2\alpha_k T}) P_{U_k}  f(u_2)  \dif u_2.
\end{split}
\end{equation}
In particular, $K_{\lambda,T}$ maps $\mcl H_{k}$ to itself. Let
\Be \label{d:KTk}
K^{(k)}_{\lambda,T}:=K_{\lambda,T} \big|_{\mcl H_{k}},
\Ee
be the restriction of the 
 operator $K_{\lambda,T}$ to $\mcl H_{k}$. For any $f \in \mcl H_{k}$, which is isomorphic to a function in $L^2([0,T];\mathbb C)$, we have
\ \ \ \
\begin{equation}  \label{e:KTKDec}
\begin{split}
   K^{(k)}_{\lambda,T}f(t)
   &=\frac{-4\beta_k^2}{\alpha_k}\int_0^Te^{-(\bar{a}_k-\mi 2\lambda  { {\beta}}_k)t-(a_k+\mi 2\lambda  { {\beta}}_k)s}(e^{2\alpha_k(t\vee s) }-e^{2\alpha_k T})  f(s)  \dif s, \ \ \ t \in [0,T].
\end{split}
\end{equation}
\vskip 3mm

Recall from \eqref{e:SigKT} that $\sigma(K_{\lambda,T})$ denotes the spectrum of $K_{\lambda,T}$. Denote by $\sigma(K_{\lambda,T}^{(k)})$ the spectrum of $K_{\lambda,T}^{(k)}$. It is easy to see that 
 \begin{equation}\label{spctra}
      \sigma(K_{\lambda,T} )=\bigcup_{k=1}^d \sigma\left(K_{\lambda,T}^{(k)}\right)=\set{\gamma_j^{(k)}:\,j=1,2,\dots,\, k=1,\dots,d },
 \end{equation}
 where some of $\gamma_j^{(k)}$ may take the same value. Indeed,  the eigenvalues and eigenfunctions of $K^{(k)}_{\lambda,T}$ are clearly those of $ K_{\lambda,T}$, which
 implies $\sigma(K_{\lambda,T} ) \supset \bigcup_{k=1}^d \sigma\left(K_{\lambda,T}^{(k)}\right)$.  On the other hand, if $\gamma$ and $f$ are  an eigenvalue and the corresponding eigenfunction of $K_{\lambda,T}$, respectively (that is $K_{\lambda,T} f=\gamma f$), then by \eqref{e:DirSumF}-\eqref{d:KTk}, 
 \begin{equation*}
\oplus_{k=1}^{d} K^{(k)}_{\lambda,T} P_{U_k}f=\oplus_{k=1}^{d} \gamma P_{U_{k}} f,
 \end{equation*}
 from which it follows that $\gamma$ must be an eigenvalue of some $K^{(k)}_{\lambda,T}$. This shows that  $\sigma(K_{\lambda,T} ) \subset \bigcup_{k=1}^d \sigma\left(K_{\lambda,T}^{(k)}\right)$. 

\vskip 3mm
The next proposition identifies a complex Sturm-Liouville problem governing the eigenfunctions of $K_{\lambda,T}^{(k)}$.
This result will enable us to obtain useful bounds on the eigenvalues of this operator. 
\begin{prop}
Let $1 \le k \le d$ be such that $\beta_k\neq 0$. Then ${\rm Ker}\left(K_{\lambda,T}^{(k)}\right)=\set{0}$. If $f \in \mcl H_{k}$ is nonzero and
$ K_{\lambda,T}^{(k)} f=\gamma f $, then
$f$ is a solution of the differential equation on $[0,T]$:
 \begin{equation}\label{kfequa}
    f'' -2\mi (1+2\lambda   )\beta_k  f' -\left[(1+2\lambda )^2\beta_k^2-\left(\frac{8\beta^{2}_k}{\gamma}-\alpha_k^2\right)\right] f=0,
 \end{equation}
subject to the separated boundary conditions
\begin{align}\label{kf}
\left\{
      \begin{array}{ll}
      [\alpha_k-\mi(1+2\lambda )\beta_k] f(0)+f'(0)=0, \\
     f(T)=0.
      \end{array}
\right.
\end{align}
\begin{proof} Let   {$\theta_{k}=a_k+2\mi \lambda  { {\beta}}_k $} and for $f \in \mcl H_{k}$ let $\tl K^{(k)}_{\lambda,T} f=-\frac{\alpha_{k}}{4 \beta^{2}_{k}}{K^{(k)}_{\lambda,T}}f.$
  It follows from \eqref{e:KTKDec} that
\begin{align}
  \tl K_{\lambda,T}^{(k)} f(t)&=e^{-\bar{\theta}_{k} t}(e^{2\alpha_k t}-e^{2\alpha_k T})\int_0^t e^{ -\theta_{k} s}f(s)\dif s\nonumber        \\
  &+ e^{ -\bar{\theta}_{k} t}\int_{t}^T e^{-\theta_{k} s}(e^{2\alpha_k s}-e^{2\alpha_k T})f(s)\dif s.\label{kf11}
\end{align}
Thus, $\tl K_{\lambda,T}^{(k)} f(t  ),\,t\in[0,T]$ is absolutely continuous. By differentiating both sides of (\ref{kf11}) with respect to $t$, we obtain
\begin{equation*}\label{kfpri}
 ( \tl K_{\lambda,T}^{(k)} f)'=-\bar{\theta}_{k} (\tl K_{\lambda,T}^{(k)}f)+  2\alpha_k e^{\theta_{k} t}\int_0^{t} e^{-\theta_{k}  s}f(s)\dif s,
\end{equation*} which implies that
\begin{equation}\label{mm}
 2\alpha_k e^{\theta_{k} t}\int_0^{t} e^{-\theta_{k}  s}f(s)\dif s=  ( \tl K_{\lambda,T}^{(k)} f)'+ \bar{\theta}_{k} (\tl K_{\lambda,T}^{(k)}f).
  \end{equation}
Similarly, $(\tl K_{\lambda,T}^{(k)} f)'(t),\,t\in[0,T]$ is absolutely continuous and we have that
\begin{align}
 ( \tl K^{(k)}_{\lambda,T} f(t))''&=-\bar{\theta}_{k} (\tl K_{\lambda,T}^{(k)}f(t))'+  2\alpha_k\theta_{k} e^{ \theta_{k} t}\int_0^{t} e^{-\theta_{k}  s}f(s)\dif s+2\alpha_k f(t). \label{kfpri2}
\end{align}
Substituting (\ref{mm}) into (\ref{kfpri2}), we have that
\begin{align}
  (\tl K^{(k)}_{\lambda,T} f)'' 
  &=(\theta_{k}-\bar{\theta}_{k})( \tl K_{\lambda,T}^{(k)} f)'+\abs{\theta_{k}}^2\tl K_{\lambda,T}^{(k)}f+2\alpha_k f(t).\label{bbb}
  \end{align}

\vskip 3mm 
Since $\alpha_k<0 $, if for $f \in \mcl H_k$, $ K^{(k)}_{\lambda,T}f=0 $ we have from (\ref{bbb}) immediately  that $f(t)\equiv 0$.
 This shows that ${\rm Ker}\left(K_{\lambda,T}^{(k)}\right)=\set{0}$.

If instead $f \in \mcl H_k$ is such that $K^{(k)}_{\lambda,T} f=\gamma f $ with $\gamma\neq 0$, then substituting $t=0$ in (\ref{mm}) and $t=T$ in (\ref{kf11}) gives the separated boundary condition (\ref{kf}). Furthermore (\ref{bbb}) implies that (\ref{kfequa}) holds.
\end{proof}
\end{prop}

Now we solve the Sturm-Liouville problem \eqref{kfequa} and \eqref{kf} {for a fixed $k$ with $\beta_k\neq 0$}, and provide  estimates for the eigenvalues of $K_{\lambda,T}$.  
We consider three cases corresponding to different range of values of $\gamma$. 

Case I: $\gamma> \frac{8 \beta^{2}_{k}}{\alpha_k^2}$. Let  $v= \sqrt{\alpha_k^2 -\frac{8\beta^{2}_k}{\gamma}}$. Note that  {$v>0$}. The general solution of equation (\ref{kfequa}) in this case is
\begin{equation*}
 f(s)=e^{\mi (1+2\lambda    )\beta_k s} (c_1 e^{v s} +c_2 e^{-v s} ) ,
\end{equation*}
where $c_1,c_2$ are constants. The boundary condition (\ref{kf}) gives
\begin{equation*}\label{bdd0 11}
\left\{
      \begin{array}{ll}
        (\alpha_k +v) c_1 +(\alpha_k-v) c_2=0,  \\
     c_1 e^{v T}+    c_2 e^{-v T} =0.
      \end{array}
\right.
\end{equation*}  
It is easy to check that since $v>0$ and $\alpha_k<0$, we must have $c_1=c_2=0$ and thus $f=0$ is the only solution in this case.


Case II:  $\gamma= \frac{8 \beta^{2}_{k}}{\alpha_k^2}$. The general solution of equation (\ref{kfequa}) in this case is
\begin{equation*}\label{solutionkr 1}
 f(s)=e^{\mi (1+2\lambda  )\beta_k s}(c_1 +c_2 s ),
\end{equation*}
where $c_1,c_2$ are constants. The boundary condition (\ref{kf}) gives
\begin{equation*}\label{bdd0 1}
\left\{
      \begin{array}{ll}
        \alpha_k  c_1 + c_2=0,  \\
     c_1+  T c_2  =0.
      \end{array}
\right.
\end{equation*}
Since $\alpha_k T-1<0$, once more we must have $c_1=c_2=0$.

Case III:  $\gamma<\frac{8 \beta^{2}_{k}}{\alpha_k^2}$. Let $\omega=\sqrt{\frac{8\beta^{2}_k}{\gamma}-\alpha_k^2}$. Note that  $\omega>0$. The general solution of equation (\ref{kfequa}) is
\begin{equation}\label{solutionkr}
 f(s)=c_1 e^{\mi ((1+2\lambda   )\beta_k+\omega)s}+c_2 e^{\mi ((1+2\lambda   )\beta_k-\omega)s},
\end{equation}
where $c_1,c_2$ are constants. The boundary condition (\ref{kf}) gives
\begin{equation}\label{bdd0}
\left\{
      \begin{array}{ll}
        (\alpha_k +\mi  \omega )c_1 +(\alpha_k -\mi \omega) c_2=0,  \\
     e^{\mi \omega T}c_1+  e^{ -\mi \omega T} c_2  =0.
      \end{array}
\right.
\end{equation}
In order to have constants $c_1$ and $c_2$ such that $c_1^2+c_2^2\neq 0$,
we need
$$(\alpha_k+\mi\omega) e^{-\mi \omega T}=(\alpha_k-\mi\omega) e^{\mi \omega T},$$
which can be rewritten as 
\begin{equation}\label{ew}
\frac{\omega}{\alpha_k}=\tan wT.
\end{equation}
Observe that $\tan wT$ is a periodic function with  period $\frac{\pi}T$, whose restriction on $(-\frac{\pi}{2T},\frac{\pi}{2T})$ is a function crossing the origin and tending to $-\infty$ and $+\infty$ at $-\frac{\pi}{2T}$ and $\frac{\pi}{2T}$ respectively, and that $\frac{\omega}{\alpha_k}$ is a linear function crossing the origin and lying in the second and fourth quadrants. {Thus there is a unique intersection point between $\frac{\omega}{\alpha_k}$ and $\tan wT$ in each period $[\frac{(2j-1) \pi}{2T},\frac{(2j+1) \pi}{2T}]$ with $j \in \mathbb Z$.} 
  Since $\omega>0$, we only need to consider  positive solutions to \eqref{ew}. Denoting these solutions (in increasing order) as {$\omega^k_j(T)$ with $j=1,2,...$}, we have
\begin{equation}\label{bdx}
\frac{\pi}{2T}<\omega_1^k(T)<\frac{3\pi}{2T}<{\omega}_2^k(T)<\frac{5\pi}{2T}<\omega_3^k(T)<\frac{7\pi}{2T}< {\omega}_4^k(T)<\dots.
\end{equation}
Hence, the spectrum of  $K_{\lambda,T}^{(k)}$ is given as
\begin{equation}\label{e: Gamma jk}
\gamma_1^{(k)} \ge \gamma_2^{(k)} \ge \gamma_3^{(k)} \ge ... \ge 0 \ \ {\rm with }\ \ \gamma_j^{(k)}= \frac{8\beta^{2}_k}{\alpha_k^2+(\omega_j^k(T))^2}.
\end{equation}
From (\ref{spctra}), we now see that the {largest} eigenvalue $\gamma_1$ of $K_{\lambda,T}$ is given as
\begin{equation}\label{bound total}
     \gamma_1=\max_{k=1,...,d}\set{\frac{8\beta_k^2}{\alpha_k^2+(\omega_1^k(T))^2}},
\end{equation}
and that 
\begin{equation}  \label{e:SpeDec}
 \sigma(K_{\lambda,T})=  \set{\frac{8\beta_k^2}{\alpha_k^2+(\omega_j^k(T))^2}:\,k=1,\dots,d;\, j=  {1,2},\dots }.
\end{equation} 
For given $k$ and $j$, the solution of \eqref{bdd0} associated with $\omega = \omega_j^k(T)$ is given by 
 $c_1=\frac{e^{\mi \theta_{j}^{k}}}{2\mi},\,c_2=\bar{c}_1 $  with $$e^{\mi \theta_{j}^{k}}=\frac{-\alpha_k +\mi \omega_j^k(T)}{\sqrt{\alpha_k^2+(\omega_j^k(T))^2}}.$$ 
 Substituting these values in \eqref{solutionkr} we see that the eigenfunction (up to multiplying a constant) associated with the eigenvalue  $\gamma_j^{(k)}$ is given as
\begin{align}\label{e: Eigenfunc}
f_j^{k}(u)=e^{\mi (1+2\lambda)\beta_k u}\sin (\omega_j^k(T) u+\theta_j^{k})\cdot U_k \in L^{2}([0,T]; \mathbb C^{d}).
\end{align}

\section{Proof of the main result} \label{s:MResult}

\subsection{The calculation of $\tilde \Lambda$}  \label{s:ProMai}
In this section we characterize the domain of $\tilde \Lambda$ and give a formula for $\tilde \Lambda(\lambda)$ for $\lambda$ in the domain.

\begin{lem}\label{c.Lem 5.1}
Let $T>0, j\in \mathbb N, \alpha<0$. Assume that $\frac{(2j-1)\pi}{2T}<\omega<\frac{(2j+1)\pi}{2T}$ and that $\theta \in [0,\frac{\pi}{2})$ satisfies $e^{\mi \theta}=\frac{-\alpha+\mi \omega}{\sqrt{\alpha^2+\omega^2}}$. Then 
\begin{align*}
  \frac12 (1-\frac{1}{\pi})T \le \int_0^T \sin^2(\omega u +\theta) \dif u \le \frac12 \big[(1+\frac{1}{\pi})T -\frac{1}{\alpha}\big].
\end{align*}
\end{lem}

\begin{proof}
Since $e^{\mi \theta}=\frac{-\alpha+\mi \omega}{\sqrt{\alpha^2+\omega^2}}$ and  $\frac{(2j-1)\pi}{2T}<\omega<\frac{(2j+1)\pi}{2T}$, we have 
\begin{align*}
\frac{\sin 2\theta}{2\omega}=\frac{\cos\theta \sin\theta}{\omega}=\frac{-\alpha}{\alpha^2+\omega^2 }\in (0,-\frac{1}{\alpha}]
\end{align*}
and 
\begin{align*}
\abs{\frac{\sin 2(\omega T+\theta)}{2\omega}}\le \frac{1}{2\omega}\le\frac {T} {(2j-1)\pi}\le \frac {T} {\pi}.
\end{align*} 
A straightforward calculation gives 
\begin{align}\label{int sin}
\int_0^T \sin^2(\omega u +\theta) \dif u&=\frac12 ( T- \frac{\sin 2(\omega T+\theta)- \sin 2\theta}{2\omega}),
\end{align} 
which, together with the previous two relations, immediately yields the inequality in the lemma.
\end{proof}

\begin{lem}\label{c.Lem 5.2}
{Assume the same conditions as in Lemma~\ref{c.Lem 5.1}.
Then
\begin{align*}
\abs{\int_0^T e^{-\alpha u} (e^{2\alpha u} - e^{2\alpha T})\sin(\omega u + \theta)
\dif u}\le   \frac{-2\alpha e^{\alpha T}}{ {\alpha^2+\omega^2}}+ \frac{-2\omega \alpha}{(\alpha^2+\omega^2)^{\frac32}}.
\end{align*}}
%
%
%
%
%
\end{lem}

\begin{proof}
By the following equalities   
\begin{align*}
\int_0^T e^{-\alpha u}  \sin (\omega u+\theta)\dif u=\frac{e^{-\alpha T} \sin \omega T}{\sqrt{\alpha^2+\omega^2}}, \quad \int_0^T e^{\alpha u}  \sin (\omega u+\theta)\dif u=\frac{-e^{\alpha T} \sin (\omega T+2\theta) +\sin 2\theta}{\sqrt{\alpha^2+\omega^2}},
\end{align*} 
we have 
\begin{align*}
\int_0^T e^{-\alpha u} \big(e^{2\alpha u}  -e^{2\alpha T}\big) \sin (\omega u+\theta)\dif u 
&=\frac{ - 2\cos \theta \sin (\omega T+\theta)e^{\alpha T} +\sin 2\theta}{\sqrt{\alpha^2+\omega^2}}.
\end{align*}
On the other hand, the assumption $e^{\mi \theta} = \frac{-\alpha+\mi \omega}{\sqrt{\alpha^2+\omega^2}}$ implies $\cos \theta=-\frac{\alpha}{\sqrt{\alpha^{2}+\omega^{2}}}$  and $\sin \theta=\frac{\omega}{\sqrt{\alpha^{2}+\omega^{2}}}$. This and the previous relation 
 immediately yield the inequality in the lemma.  
\end{proof}

\begin{lem} \label{c.Lem 5.3}
Let $\theta\le \min\limits_{k=1,...,d}\set{\frac{\alpha_k^2 }{8\beta_k^2}}$. For sufficiently large $T$ we have
\begin{align*}
 \Ll G_{\lambda,T}^x, ({\rm Id}-\theta K_{\lambda,T})^{-1}G_{\lambda,T}^x\Rr_{L^{2}([0,T]; \R^{d})}\le 
 \frac{1}{\theta^2 ( 1-\frac{1}{\pi})}\sum_{k=1}^{d} |\alpha_{k}| {\abs{\innp{x,\,U_k}_{\mathbb C^{d}}}^2}  .
 \end{align*}
\end{lem}

\begin{proof}
By (\ref{key point 1}), 
\begin{align*}
 \Ll G_{\lambda,T}^x, ({\rm Id}-\theta K_{\lambda,T})^{-1}G_{\lambda,T}^x\Rr_{L^{2}([0,T]; \R^{d})}&=\sum_{k=1}^{d}\sum_{j=1}^\infty \frac{\abs{\innp{G_{\lambda,T}^x,\, f_j^{k}}_{L^{2}([0,T]; \C^{d})}}^2 }{(1-\theta \gamma_j^{k})\norm{ f_j^{k}}^2_{L^{2}([0,T]; \C^{d})}},
\end{align*} 
{where $\{f_j^{k}, j \in \mathbb{N}, k= 1, \ldots d\}$ is a complete orthonormal system of eigenvectors associated with eigenvalues $\{\gamma_j^{(k)}\}$. Note that when $\beta_k \neq 0$, $f_j^{k}$ is given by \eqref{e: Eigenfunc}.}
It follows from  (\ref{e:GxT}) that, {when $\beta_k \neq 0$},  
\begin{align*}
&\abs{\innp{G_{\lambda,T}^x,\, f_j^{k}}_{L^{2}([0,T]; \C^{d} }}\\
&=   \abs{\frac{4\beta_k^2}{\alpha_k} \innp{x,\,U_k}_{\mathbb C^{d}}  \int_0^T  e^{-\alpha_k u} \big(e^{2\alpha_k T}-e^{2\alpha_k u} \big) \sin (\omega_j^k(T) u+\theta_j^{k})\dif u}\\
&\le \frac{1}{2\theta} \abs{ {\alpha_k} \innp{x,\,U_k}_{\mathbb C^{d}} \cdot \Big[ \frac{-2\alpha_k e^{\alpha_k T}}{ {\alpha_k^2+(\omega_j^k(T))^2}}+ \frac{-2\omega_j^k(T) \alpha_k}{(\alpha_k^2+(\omega_j^k(T))^2)^{\frac32}}\Big]  },
\end{align*}
where the last inequality is by Lemma~\ref{c.Lem 5.2}. {Note that the above inequality is trivially true when
$\beta_k=0$.}
By (\ref{e: Gamma jk}), 
\begin{align*}
\frac{1}{1-\theta \gamma_j^{k}}\le \frac{1}{1- \frac{\alpha_k^2 }{8\beta_k^2} \gamma_j^{k}}=\frac {\alpha_k^2+(\omega_j^k(T))^2}{(\omega_j^k(T))^2}.
\end{align*}
The above two  bounds and Lemma~\ref{c.Lem 5.1} imply that 
\begin{align*}
& \Ll G_{\lambda,T}^x, ({\rm Id}-\theta K_{\lambda,T})^{-1}G_{\lambda,T}^x\Rr_{L^{2}([0,T]; \R^{d})}\\
& \le \frac{2}{\theta^2 ( 1-\frac{1}{\pi})}\sum_{k=1}^{d}\abs{ {\alpha_k} \innp{x,\,U_k}_{\mathbb C^{d}}}^2 \sum_{j=1}^\infty   \frac{1}{T} \Big[ \frac{ -\alpha_k e^{\alpha_k T}}{{\omega_j^k(T)} \sqrt{\alpha_k^2+(\omega_j^k(T))^2}}+ \frac{- \alpha_k}{(\alpha_k^2+(\omega_j^k(T))^2) }\Big]^2\\
&\le \frac{2}{\theta^2 ( 1-\frac{1}{\pi})}\sum_{k=1}^{d}\abs{ {\alpha_k} \innp{x,\,U_k}_{\mathbb C^{d}}}^2 \sum_{j=1}^\infty   \frac{1}{T} \Big[ \frac{ -2T\alpha_k e^{\alpha_k T}}{{\pi} \sqrt{\alpha_k^2+(\omega_j^k(T))^2}}+ \frac{- \alpha_k}{(\alpha_k^2+(\omega_j^k(T))^2) }\Big]^2
\end{align*}
where the last inequality is from  \eqref{bdx}. As $T\to \infty$, the inner sum in the last line has the following limit:
\begin{align*}
\frac{1}{\pi}\int_{0}^{\infty}\frac{\alpha_k^2}{(\alpha_k^2+t^2)^2 } \dif t =\frac{1}{4\abs{\alpha_k}}.
\end{align*}
Indeed, thanks to \eqref{bdx} and recalling $\alpha_{k}<0$,  we have 
 \begin{align*}
 \lim_{T\to \infty} 2T e^{\alpha_k T}&=0,\;\;
 \lim_{T\to \infty} \sum_{j=1}^\infty   \frac{1}{T}  \frac{ \alpha_k}{  {\alpha_k^2+(\omega_j^k(T))^2}}=\frac{\pi}{2}, \\
 \lim_{T\to \infty} \sum_{j=1}^\infty   \frac{1}{T}  \frac{(\alpha_k)^2}{ \sqrt{\Big(\alpha_k^2+(\omega_j^k(T))^2\Big)^3}}&=1, \;\;
 \lim_{T\to \infty} \sum_{j=1}^\infty   \frac{1}{T}  \frac{(\alpha_k)^2}{\Big(\alpha_k^2+(\omega_j^k(T))^2\Big)^2}=\frac{1}{4\abs{\alpha_k}}.
 \end{align*}
 Hence, when $T$ is large enough, we have 
\begin{align*}
\sum_{j=1}^\infty   \frac{1}{T} \Big[ \frac{-2T\alpha_k  e^{\alpha_k T}}{{\pi} \sqrt{\alpha_k^2+(\omega_j^k(T))^2}}+ \frac{- \alpha_k}{(\alpha_k^2+(\omega_j^k(T))^2) }\Big]^2\le \frac{1}{2\abs{\alpha_k}},
\end{align*}which yields the desired upper bound.
\end{proof}

Recall the interval $\mcl D$ introduced in Theorem \ref{t:FT}.
\begin{lem}  \label{l:CalTlLam}
	For $\lambda \in \mcl D$
\begin{equation}\label{Lambda closed form}
\tilde \Lambda(\lambda)= \lim_{t\to \infty} \tilde \Lambda_t(\lambda)= -\frac 12\sum_{k=1}^{d}   \left(\sqrt{ \alpha_k^2- 4\lambda(1+\lambda) \beta_k^2} +\alpha_k\right) .
\end{equation}
Moreover, when $\lambda \in \mcl D^c$,
we have 
$$\lim_{T\to \infty}\tilde \Lambda_T(\lambda)=\infty.$$
\end{lem}
\begin{proof}
For $\lambda \in \R$, denote $\theta=\frac 12 \lambda(1+\lambda)$. Note that $\lambda \in \mcl D$ if and only if 
$\theta \le \min\limits_{k=1,...,d}\set{\frac{\alpha_k^2 }{8\beta_k^2}}$.
It follows from Theorem \ref{t:ExpIntZ} that when $\theta = \frac{1}{\gamma_1} \ge \min\limits_{k=1,...,d}\set{\frac{\alpha_k^2+(\omega_1^k(T))^2}{8\beta_k^2}}$,
\begin{align}
& \frac{1}{T}\log \E^{x} \exp\left(\theta \int_{0}^{T} |Z_{\lambda,t}|^{2} \dif t\right) =\infty,  \quad  \mbox{ for all } x \in \R^{d},  \mbox{ and }\quad  T>0,  \label{e:ExExpZInf}
\end{align} 
and thus \eqref{e:ExExpZInf} also holds true with  $\E^{x}$ replaced by $\E^{\mu}$. This shows that
\begin{equation}  \label{e:ExExpZInf-1}
\tl \Lambda_{T}(\lambda)=\infty \ \ \quad \ {\rm whenever} \ \ \frac 12 \lambda(1+\lambda) \ge \min\limits_{k=1,...,d}\set{\frac{\alpha_k^2+(\omega_1^k(T))^2}{8\beta_k^2}}.
\end{equation} 
The second relation in the lemma is now immediate.

We now prove the first statement in the lemma. 
We need to show that when $\theta\le \min\limits_{k=1,...,d}\set{\frac{\alpha_k^2 }{8\beta_k^2}}$
\begin{align*}
	\frac{1}{T}\log \E^{\mu} \exp\left(\theta \int_{0}^{T} |Z_{\lambda,t}|^{2} \dif t\right)
\end{align*}
converges to the right side of  \eqref{Lambda closed form}.
By Theorem  \ref{t:ExpIntZ}
\begin{align*}
&	\E^{\mu} \exp\left(\theta \int_{0}^{T} |Z_{\lambda,t}|^{2} \dif t\right)\\
&=	\frac{1}{\sqrt{\mathrm{det}({\rm Id}-\theta K_{\lambda,T})}}  \\
    & \quad \times \int_{\R^d} \exp\left[\theta S_{0}^x(T)-\frac \theta 2  \tr(K_{\lambda,T})+\frac{\theta^2}2 \Ll G_{\lambda,T}^x, ({\rm Id}-\theta K_{\lambda,T})^{-1}G_{\lambda,T}^x\Rr_{L^{2}([0,T]; \R^{d})}\right] \mu(dx),
\end{align*}
Thus, using \eqref{eq:soxt},
\begin{align*}
\frac{1}{T}\log \E^{\mu} \exp\left(\theta \int_{0}^{T} |Z_{\lambda,t}|^{2} \dif t\right)	= I_1 + I_2+ I_3
\end{align*}
where
$I_{3}=-\frac{1}{2T}  \log \left[\mathrm{det}({\rm Id}-\theta K_{\lambda,T})\right]$,
\begin{align*}
I_1 &= - \frac{\theta}{2T} \tr(K_{\lambda,T})+\frac{\theta}{T}\int_{0}^{T} \tr\left[N^{'} e^{uM} {N}\right] (T-u) \dif u,
\end{align*}
and
\begin{align*}
	I_2 &= 
	\frac{1}{T}\log \int_{\R^d} \exp\left[\left|\left(\int_0^T e^{sM} \dif s\right)^{\frac 12}       N x\right|^{2}\theta
	+\frac{\theta^2}2 \Ll G_{\lambda,T}^x, ({\rm Id}-\theta K_{\lambda,T})^{-1}G_{\lambda,T}^x\Rr_{L^{2}([0,T]; \R^{d})}\right] \mu(dx)
\end{align*}
{Observe that the following relations hold: 
\Bes
\int_{0}^{T} e^{uM} \dif u=M^{-1} \left(e^{MT}-{\rm Id}\right),  \ \ \ \ \ \int_{0}^{T} e^{uM} (T-u) \dif u=M^{-2}\left(e^{MT}-{\rm Id}\right)-T M^{-1}.
\Ees
Since $M$ is negative definite, $\|e^{Mt}\| \le 1$ for all $t \ge 0$, and thus
$$\lim_{T \rightarrow \infty} 
\frac{1}{T}\int_{0}^{T} \tr\left[N^{'} e^{uM} {N}\right] (T-u) \dif u
=-{\rm tr}(N^{'} M^{-1} N).$$ 
On the other hand, \eqref{e:TrKTMer} implies $\lim_{T \rightarrow \infty}\frac{1}T  \tr(K_{\lambda,T})=-2 {\rm tr}(N^{'} M^{-1} N)$, which, together with the previous limit, yields
$$\lim_{T \rightarrow \infty} I_{1}=0.$$}

Recalling that the density of $\mu$ is given by \eqref{eq:density}, the relation (\ref{e: S01}) and Lemma~\ref{c.Lem 5.3} imply that 
\begin{align*}
I_2 &\le \frac{1}{T}\log \int_{\R^d} (2\pi)^{-\frac{d}{2}}|\mathrm{det} (M)|^{\frac{1}{2}}  \exp \left[\sum_{k=1}^d \left(\frac{\abs{\alpha_k}}{4} + \frac{\abs{\alpha_k}}{2(1-\frac{1}{\pi})  }\right) \abs{\innp{x,U_k}_{\C^d}}^2\right] \exp\left(\frac{x'M x}2\right)\dif x \\
& = \frac{1}{T}\log \int_{\R^d} (2\pi)^{-\frac{d}{2}}|\mathrm{det} (M)|^{\frac{1}{2}} \exp \left[  (\frac{3}{4}-\frac{1}{2(1-\frac{1}{\pi})  })  \sum_{k=1}^d \alpha_k |{y}_k|^2\right] \dif x,
\end{align*}
where the last inequality uses the fact that $U^* M U=\mbox{diag}\{2 \alpha_1,...,2 \alpha_d\}$ by \eqref{eq:spect},
and $\exp\left(\frac{x'M x}2\right)=\exp\left(\sum_{k=1}^d \alpha_k |y_k|^2\right)$ with $y_k=\innp{x,U_k}_{\C^d}$. Since $\frac{3}{4}-\frac{1}{2(1-\frac{1}{\pi})  }>0$ and $\alpha_k<0$,  we have that, $I_2<\frac{c}{T}$ with a constant $c>0$ depending only on $M$. Thus $\lim_{T\to \infty} I_{2}=0.$

For $I_{3}$, by Theorem \ref{t:ExpIntZ} and \eqref{e:SpeDec} we have 
\Bes
\begin{split}
I_{3}&=-\frac{1}{2T}  \sum_{j=1}^{\infty} \sum_{k=1}^{d} \log (1-\theta \gamma^{(k)}_{j}) \\ 
&=-\frac{1}{2T}  \sum_{j=1}^{\infty} \sum_{k=1}^{d} \log \left(1-\frac{8 \theta\beta^{2}_k}{\alpha_k^2+(\omega_j^k(T))^2}\right) \\ 
& \rightarrow -\frac 12 \int_{0}^{\infty} \sum_{k=1}^{d} \log \left(1-\frac{8 \theta\beta^{2}_k}{\alpha_k^2+\pi^{2} x^2}\right) \dif x 
\end{split}
\Ees
as $T \rightarrow \infty$, where the convergence on the last line follows from \eqref{bdx}.  

The result now follows on observing that  for $b>0$ and $b\ge a$,
\begin{align*}
\int_0^{\infty} \log (1 - \frac{a}{b+x^2}) dx = \big(\sqrt{b-a}-\sqrt{b}\, \big) \pi.
\end{align*}
\end{proof}

\subsection{Proofs of Theorems \ref{t:EPR} and \ref{t:FT}}  \label{ss:PLDP}
For  $n \in \N$, define
$$\tau_n=\inf\{t \ge 0: |X_{t}| \ge n\}, \ \ \ \ X^{n}_{t}=X_{t} 1_{\{t  \le \tau_n\}}.$$
Clearly, 
$\lim_{n \rightarrow \infty} \tau_n=\infty$  a.s.
For $\lambda \in \R$, let
\begin{equation*}
 \dif {W}_{t}= \dif {B}_{t} - \lambda       N X^{n}_t \dif t,
\end{equation*}
and
\begin{align*}
\mathcal E^n_t(\lambda)&=\exp\set{-\frac{ \lambda^2 }{2}\int_0^{t} |      NX^{n}_s|^{2}  \dif s +  \lambda \int_0^{t} (      NX^{n}_s)^{'} \dif B_{s}}\\
&= \exp\set{-\frac{ \lambda^2 }{2}\int_0^{t \wedge \tau_n} |      NX_s|^{2}  \dif s +  \lambda \int_0^{t\wedge \tau_n} (      NX_s)^{'} \dif B_{s}}.
\end{align*}
Recall $\tl \Lambda$ and $\Lambda$ in Section {\ref{s:Epr} }. 
Denote the probability space on which the stationary process $X_t$  and the Brownian motion $B_t$ are defined as $(\Omega, \mcl F, \PP)$. 
Let $\mcl F_t$ be the filtration generated by $X_0$ and $\{B_t\}$, namely $\mcl F_t = \sigma\{X_0, B_s, 0\le s\le t\}$.
Let
\begin{align*}
\Lambda^n_t(\lambda)
&=\frac{1}{t}\log \E\exp\set{\lambda \int_{0}^{t\wedge \tau_n}\,(      NX_s)^{'} \dif B_s  +\frac\lambda 2 \int_0^{t\wedge \tau_n} |      NX_s|^{2}\dif s}.
\end{align*}
We now  prove \eqref{e:TilD=D}.
\begin{prop} \label{p:PropTilD1}
	We have $\mathcal D_{\tilde \Lambda}=\mathcal D_{\Lambda}$ and $\tilde \Lambda=\Lambda$.
\end{prop}



\begin{proof} 
\emph{\underline{Step 1. The relation between $X_{t}$ and $Y_{\lambda, t}$}:}
Since $|X^{n}_{t}| \le n$ for all $t$, Novikov's condition (cf. \cite[3.5.D]{karshr}) clearly holds for the exponential supermartingale $\mathcal E^n_t(\lambda)$. Thus Girsanov theorem yields that $(W_{t})_{t \ge 0}$ is a standard $d$-dimensional Brownian motion under the measure $\PP^n_{W}$
which is uniquely determined by
\begin{equation} \label{e:CMG}
  \frac{\dif \PP^n_{W}}{\dif \PP}\big |_{\mathcal F_{t}}=\mathcal E^n_t(\lambda), \ \ \ \ \ \ t \ge 0.
\end{equation}
and  equation \eqref{eq:redu} can be rewritten as
\begin{equation*} 
\begin{split}
\dif X_{t}&={A} X_{t}\dif t+\lambda  NX_{t} 1_{\{t \le \tau_n\}} \dif t+\dif W_t.
\end{split}
\end{equation*}
If one defines $Y_{\lambda, t}$ on $(\Omega, \mcl F, \PP^n_{W})$ as 
\begin{equation}  \label{langevin new-1-1}
\begin{split}
\dif Y_{\lambda, t}=D_\lambda Y_{\lambda, t}\dif t + \dif W_t, \; Y_{\lambda, 0} = X_0,
\end{split}
\end{equation} 
then $\{Y_{\lambda, t}\}$ (under $\PP^n_{W}$) has the same law as the process in \eqref{langevin new} under the stationary measure
$\PP^{\mu}$ considered there. Also, $Z_{\lambda, s} = N Y_{\lambda, s} = N X_s$ for $0\le s \le \tau_n \wedge t$.
Thus
\begin{align*}
	  \Lambda^n_t(\lambda)&= \frac{1}{t}\log \E\exp\set{\lambda \int_{0}^{t\wedge \tau_n}\,(      NX_s)^{'} \dif B_s  +\frac\lambda 2 \int_0^{t\wedge \tau_n} |      NX_s|^{2}\dif s}\\
	   &=  \frac{1}{t}\log \E^{\PP^n_W}\exp\set{\frac12\lambda(1+\lambda)\int_{0}^{t \wedge \tau_n}\,  |NX_s|^{2} \dif s}\\
& = \frac{1}{t}\log \E^{\mu}\exp\set{\frac12\lambda(1+\lambda)\int_{0}^{t \wedge \tau_n}\,  |Z_{\lambda,s}|^{2} \dif s}.
 \end{align*}
By Fatou's lemma and monotone convergence theorem, we  have
\begin{equation}   \label{e:Fatou}
\Lambda_t(\lambda) \le \lim_{n \rightarrow \infty} \Lambda^n_t(\lambda)=\tilde \Lambda_t(\lambda).
\end{equation}

\emph{\underline{Step 2. $\mathcal D^{\circ}_{\tilde \Lambda} \subset \mathcal D^{\circ}_{\Lambda}$}:}
Recall from Lemma \ref{l:CalTlLam} the domain of $\tilde \Lambda$ is a finite closed interval, we denote this interval  by $[a,b]$ and see that $-\infty<a<0<b<\infty$. Clearly $0  \in \mathcal D_{\tilde \Lambda} \cap \mathcal D_{\Lambda}$. 

Suppose now that $\lambda \in (0,b)$. Fix  a $\bar \lambda
\in (\lambda,b)$. We have, for $t>0$,
\begin{equation}
\Lambda^n_t(\bar \lambda)=\frac 1t \log \E^{\mu}\exp\set{\frac12\bar \lambda(1+\bar \lambda)\int_{0}^{t \wedge \tau_n}\, |Z_{\lambda, s}|^{2} \dif s} \le \tilde \Lambda_t(\bar \lambda)<\infty.
\end{equation}
Because $0<\lambda<\bar \lambda$, the above inequality implies a uniform integrability (with respect to $n$) for
$\exp\set{\lambda \int_{0}^{t\wedge \tau_n}\,(      NX_s)^{'} \dif B_s  +\frac\lambda 2 \int_0^{t\wedge \tau_n} |      NX_s|^{2}\dif s}$,
and thus
\begin{equation}   \label{e:Fatou1}
\begin{split}
& \ \ \ \ \ \ \E^{\mu}\exp\set{\lambda \int_{0}^t\,(      NX_s)^{'} \dif B_s  +\frac\lambda 2 \int_0^t |      NX_s|^2 \dif s } \\
&=\lim_{n \rightarrow \infty} \E^{\mu} \exp\set{\lambda\int_{0}^{t\wedge \tau_{n}}\, (      NX_s)^{'} \dif B_s  +\frac{\lambda}{2} \int_0^{t\wedge \tau_{n}} |      NX_s|^2\dif s }.
\end{split}
\end{equation}
Hence,
\begin{equation}   \label{e:Fatou1b}
\Lambda_t(\lambda)=\lim_{n \rightarrow \infty} \Lambda^n_t(\lambda)=\tilde \Lambda_t(\lambda)    \ \ \ {\rm for} \ \lambda \in (0,b).
\end{equation}
Similarly, we have
\begin{equation}   \label{e:Fatou2}
\Lambda_t(\lambda)=\lim_{n \rightarrow \infty} \Lambda^n_t(\lambda)=\tilde \Lambda_t(\lambda)    \ \ \ {\rm for} \ \lambda \in (a,0).
\end{equation}
Hence, $\mathcal D^{\circ}_{\tilde \Lambda} \subset \mathcal D^{\circ}_{\Lambda}$ and
\Be  \label{e:DomLam1}
\tilde \Lambda(\lambda)=\Lambda(\lambda) \ \ \ {\rm for}\ \  \lambda \in \mathcal D^{\circ}_{\tilde \Lambda}.
\Ee

\emph{\underline{Step 3. $\mathcal D_{\Lambda} \subset \mathcal D_{\tilde \Lambda}$}:}
 We argue via contradiction. Suppose there exists a  $\lambda^{*}
\in \mathcal D_{\Lambda}$ such that $\lambda^{*} \notin \mathcal D_{\tilde \Lambda}$, i.e. $\lambda^* \notin [a,b]$. Assume that
$\lambda^{*}>b$; the case $\lambda^{*}<a$ can be handled similarly. 
Recalling the definition of $[a,b]$ (see Lemma \ref{l:CalTlLam}), we see that $\theta^{*}=\frac 12 \lambda^{*}(1+\lambda^{*})>\min\limits_{k=1,...,d}\set{\frac{\alpha_k^2 }{8\beta_k^2}}$. Since $\lim_{T \rightarrow \infty}\omega_1^k(T)=0$ for each $k=1,...,d$, we can choose $T$ sufficiently large so that $\theta^{*}>\hat \theta$ with $\hat \theta=\min\limits_{k=1,...,d}\set{\frac{\alpha_k^2+(\omega_1^k(T))^2}{8\beta_k^2}}$. 
Also, since $\Lambda(\lambda^*)<\infty$, by choosing $T$ larger if needed, we can assume that $\Lambda_T(\lambda^*)<\infty$.
Choose an increasing sequence $\{\lambda_{n}\}_{n \ge 1}$ such that $\lambda_{n}>b$ and $\theta_{n} \uparrow  \theta^*$ with $\theta_{n}=\frac 12 \lambda_{n}(1+\lambda_{n})$. Clearly, $\lambda_{n}<\lambda^{*}$, and since $\lim \theta_n > \hat \theta$,  by \eqref{e:ExExpZInf-1} we know 
$$\lim_{n \rightarrow \infty} \tilde \Lambda_{T}(\lambda_{n})=\infty.$$
On the other hand, since $\lambda_n <\lambda^*$, by the same argument as in the proof of \eqref{e:Fatou1b}, we get
$\tilde \Lambda_{T}(\lambda_{n})=\Lambda_{T}(\lambda_{n})$ and thus 
$$\lim_{n \rightarrow \infty} \Lambda_{T}(\lambda_{n})=\infty.$$
By H\"{o}lder inequality we get $\Lambda_{T}(\lambda^{*})  \ge \frac{\lambda^{*}}{\lambda_{n}} \Lambda_{T}(\lambda_{n})$ for all $\lambda_{n}$, and hence 
$\Lambda_{T}(\lambda^{*})=\infty$. But this is a contradiction since $T$ is chosen such that $\Lambda_{T}(\lambda^{*})<\infty$.
Thus we have that $\mathcal D_{\Lambda} \subset \mathcal D_{\tilde \Lambda}$.

\emph{\underline{Step 4. $\mathcal D_{\Lambda}=\mathcal D_{\tilde \Lambda}$ and $\tl \Lambda=\Lambda$}:} From the previous two steps, we clearly see that $(a,b) \subset \mcl D_{\Lambda} \subset [a,b]$ and $\tl \Lambda(\lambda)=\Lambda(\lambda)$ for $\lambda \in (a,b)$. 
To conclude the proof, we only need to prove $\Lambda(a)=\tl \Lambda(a)$ and $\Lambda(b)=\tl \Lambda(b)$.
For a small $\e>0$, by \eqref{e:Fatou}, \eqref{e:Fatou1b} and an application of H\"{o}lder inequality to $\Lambda_{t}(b-\e)$, we have 
\Bes
\tl \Lambda_{t}(b-\e)=\Lambda_{t}(b-\e) \le \frac{b-\e} b\Lambda_{t}(b) \le  \frac{b-\e} b\tl \Lambda_{t}(b), \ \ \ t>0,  
\Ees 
which leads to  
\Bes
\tl \Lambda(b-\e)\le \frac{b-\e} b\liminf_{t \rightarrow \infty}\Lambda_{t}(b) \le \frac{b-\e} b\limsup_{t \rightarrow \infty}\Lambda_{t}(b) \le  \frac{b-\e} b\tl \Lambda(b).
\Ees 
Since $\lim_{\e \rightarrow 0+} {\tl \Lambda(b-\e)=\tl  \Lambda(b)}$ by Lemma \ref{l:CalTlLam}, we have $\Lambda(b)=\tl \Lambda(b)$. Similarly we get $\Lambda(a)=\tl \Lambda(a)$.
\end{proof}

In order to verify conditions of G\"{a}rtner-Ellis theorem, define
\begin{equation}\label{fl}
  F(\ell)= \frac 12\sum_{k=1}^{d}   \left(\sqrt{ \alpha_k^2- \ell \beta_k^2} +\alpha_k\right),\ \ \ \ \  
  \ell \in \left(-\infty,   \min_{k=1,...,d}\set{\frac{\alpha_k^2}{\beta_k^2}}\right].
\end{equation}
It is easy to check that 
\begin{align*}
F'(\ell) =- \frac 14\sum_{k=1}^{d} \frac{\beta_k^2}{\sqrt{\alpha_k^2 -\ell \beta_k^2} },\ \ \ \ \    \ell \in \left(-\infty,   \min_{k=1,...,d}\set{\frac{\alpha_k^2}{\beta_k^2}}\right).
\end{align*} 
Note that $\abs{ F'(\ell)}\to \infty$ as $\ell \to \min\limits_{k=1,...,d}\set{\frac{\alpha_k^2}{\beta_k^2}}-$.

We can now complete the proofs of our main results.

\begin{proof} [{\bf Proofs of Theorems \ref{t:EPR} and \ref{t:FT}}]
	The first relation in  
Theorem \ref{t:FT} follows immediately from Proposition \ref{p:PropTilD1} and Lemma \ref{l:CalTlLam}. Now we prove Theorem \ref{t:EPR} and the Cohen-Gallavotti symmetry properties. 
From the previous proposition, Assumption 2.3.2 of \cite{demzeit} is satisfied.
Let
$$\ell(\lambda)=4\lambda(1+\lambda).$$
Then by Proposition \ref{p:PropTilD1} and Lemma \ref{l:CalTlLam},  we have that for all $\lambda \in \mcl D_{\Lambda}$,
\begin{align}
\Lambda(\lambda)&=-F(\ell (\lambda)),\label{Lamd lamd}\\
\Lambda'(\lambda)&= ( 1+2\lambda ) \sum_{k=1}^{d} \frac{\beta_k^2}{\sqrt{\alpha_k^2 -\ell(\lambda)\beta_k^2} }.\nonumber
\end{align}
Hence, $\Lambda'(\lambda)$ exists for all $\lambda \in \mcl D_{\Lambda}^{\circ}$ and  
$$\lim_{\lambda \rightarrow \partial \mcl D_{\Lambda}} |\Lambda'(\lambda)|=\infty.$$ 
From this it follows that $\Lambda$ is a lower semicontinuous function on $\R$ which is essentially smooth in the sense of
\cite[Definition 2.3.5]{demzeit}. Thus
by G\"artner-Ellis Theorem \cite[Section 2.3]{demzeit}, we immediate obtain that  EPR $e_p(t)$ satisfies an LDP with rate function $I$ (in particular, $I$ has compact level sets) given as 
 \begin{align*}
  I(x) &=\sup\set{x\lambda +F(\ell(\lambda)):\, \lambda \in  \mcl D_{\Lambda}} \\
    &=\sup \set{x\lambda(\ell)+F(\ell):\, -1 \le \ell \le  \min_{k=1,...,d}\set{\frac{\alpha_k^2}{\beta_k^2}}},
\end{align*}
where $F(\ell)$ is given as (\ref{fl}) and $\lambda(\ell)$ is an inverse function of $
   \ell=4  \lambda(1+\lambda)$ defined as $\lambda(\ell)=\frac{\sqrt{\ell+1}-1}{2}$ for $x\ge 0$ and $\lambda(\ell)=\frac{-\sqrt{\ell+1}-1}{2}$ for $x<0$.

When $x=0$,  it is clear that $I(0)=F(-1)$ since $F(\ell)$ is strictly decreasing on $\left[ -1,\, \min\limits_{k=1,...,d}\set{\frac{\alpha_k^2}{\beta_k^2}}\right]$.

When $x\neq 0$, by differentiating the function $s(\ell):=x\lambda(\ell)+F(\ell)$, we have that the unique zero point $\ell_0(x)$ of
$s'(\ell)$ is given as the solution of the following equation:
\begin{align*}
{\abs{x}}={\sqrt{1+\ell}}\sum_{k=1}^{d} \frac{\beta_k^2}{\sqrt{\alpha_k^2 -\ell \beta_k^2} },\qquad -1 \le \ell<  \min_{k=1,...,d}\set{\frac{\alpha_k^2}{\beta_k^2}}.
\end{align*}
Substituting it into $s(\ell)$, we obtain (\ref{rate func}).  

Finally,  the  Cohen-Gallavotti symmetry properties in \eqref{Cohen symmet} are immediate from the explicit expressions 
for  $\Lambda$ and $I$ in \eqref{eq:defnlam} and \eqref{rate func}, respectively.
\end{proof}

\section{Appendix}
{\bf Proof of the claim in \eqref{e:CommQAEtAl}}
Since ${ A} \in \R^{d \times d}$ is normal, there is a \emph{real} orthogonal matrix ${ P}$ (cf. \cite[Theorem 2.5.8]{hj}) such that
\begin{equation}  \label{e:PAP}
{ P}^{'}{ A}{  P}=\mbox{diag} \{d_1, \ldots , d_m, E_1,\,E_2,\,\cdots,E_k\},
\end{equation}
where for $1\le i \le m$, $d_i$ is  a real number and for $1\le j \le k$, $E_j$ is a real $2\times 2$ matrix
of the form, 
\begin{equation*}
E_j= \begin{bmatrix}   e_j & \tilde e_j \\ - \tilde e_j & e_j  \end{bmatrix}
\end{equation*}
and $m+2k=d$,
where {$e_j \in \R, \tilde e_j \in \R$ with $\tilde e_j\neq 0$}. From \eqref{e:PAP} and assumption {\bf (A)}, we see that 
\Be  \label{e:MEigNeg}
M:=A+A^{'} {\rm \ has \ eigenvalues}\ \ \{2 d_i, 2e_j, 1\le i \le m, 1 \le j \le k\} \ {\rm and \ is \ negative \ definite}.
\Ee
Define
\begin{equation*}
{\hat A}= \mbox{diag}\{d_1, \ldots , d_m, e_1 + \mi \tilde e_1, e_1 - \mi \tilde e_1, \ldots , e_k + \mi \tilde e_k, e_k - \mi \tilde e_k\}
%
%
%
\end{equation*}
Then there
is a complex unitary matrix $U$ (cf. \cite[Theorems 2.5.3 and 2.5.8]{hj}), such that
\begin{align}
{U}^*{A}{U}&= \mbox{diag}\{ \hat A \},\label{ubu}.
\end{align}

By  our assumption in \eqref{e:AQComm} and \cite[Theorems 2.5.5 and 2.5.6]{hj} it follows that the above unitary matrix $U$ and the diffusion matrix $Q$ satisfy 
\begin{equation}\label{ububQQ}
{U}^* {Q} {U} =\mbox{diag}\set{q_1,\dots,q_d},
\end{equation} where $q_k>0$ are the eigenvalues of $Q$. 
It is easy to check using \eqref{ubub} and \eqref{ububQQ} that 
\Bes  
A^{'} Q=Q A^{'}, \quad Q^{\frac 12} A=A Q^{\frac 12}, \quad Q^{\frac 12} A'=A' Q^{\frac 12}.
\Ees 
The statement in \eqref{e:CommQAEtAl} is immediate from this.\\

{\bf Acknowledgement.} {Research of AB is supported in part by the National Science Foundation (DMS-1814894 and DMS-1853968). Research of YC is supported in part by NSFC Grants (11871079, 11961033).
Research of LX is supported in part by Macao S.A.R grant FDCT  0090/2019/A2 and University of Macau grant  MYRG2018-00133-FST.}

\end{document}